\documentclass[oneside,a4paper,reqno,12pt]{amsart}
\usepackage{amsfonts}
\usepackage{mathrsfs}
\usepackage{amssymb}
\usepackage{amsmath}
\usepackage{latexsym,amsfonts,mathrsfs,amssymb,amsmath}
\usepackage[dvips]{graphicx,color}
\usepackage{ifpdf}
\usepackage[all]{xy}
\usepackage{verbatim}
\usepackage[cmbase]{flexisym}
\usepackage{breqn}

\usepackage[body={14.5true cm,22true cm},
            ]{geometry}

\newcommand{\po}{\partial\Omega_c}
\newcommand{\ua}{\uparrow}
\newcommand{\Ua}{\Uparrow}
\newcommand{\mt}{\mathbb}

\newtheorem{theorem}{Theorem}[section]
\newtheorem{lemma}[theorem]{Lemma}
\newtheorem{proposition}[theorem]{Proposition}
\newtheorem{corollary}[theorem]{Corollary}
\newtheorem{claim}[theorem]{Claim}
\newtheorem{key lemma}[theorem]{\bf Key Lemma}
\newtheorem{sub lemma}[theorem]{\bf Sublemma}

\theoremstyle{definition}
\newtheorem{definition}[theorem]{Definition}
\newtheorem{example}[theorem]{Example}

\newtheorem{conjecture}[theorem]{Conjecture}
\newtheorem{open problem}[theorem]{Open problem}

\theoremstyle{remark}
\newtheorem{remark}[theorem]{Remark}

\newcommand{\Alex}{{\qopname\relax{no}{Alex}}}
\newcommand{\dist}{{\qopname\relax{no}{dist}}}
\newcommand{\Length}{{\qopname\relax{no}{Length}}}

\numberwithin{equation}{theorem}

\begin{document}
\pagestyle{plain}
\title{Nonnegatively Curved Alexandrov Spaces with Souls of
Codimension Two}
\author{Xueping Li}
\date{December 30, 2012}
\thanks{2010 \emph{Mathematics Subject Classification}. Primary 53C20, 53C99; Secondary 57N65.
\newline\indent \emph{Key words and phrases.} Alexandrov space, Soul,
Sharafutdinov retraction, submetry, topologically nice, space of
directions, join.}
\maketitle

\begin{abstract}
In this paper, we study a complete noncompact nonnegatively curved
Alexandrov space $A$ with a soul $S$ of codimension two. We
establish some structural results under additional regularity
assumptions. As an application, we conclude that in this case
Sharafutdinov retraction, $\pi:\ A\to S$, is a submetry.
\end{abstract}

\section*{Introduction}
We begin with the classical Soul Theorem  of Cheeger-Gromoll
(\cite{[CG]}) on complete noncompact Riemannian manifolds of
nonnegative sectional curvature:
\begin{theorem}
Let $M$ be a complete noncompact  Riemannian $n$-manifold with
sectional curvature $sec(M)\geqslant0$. Then $M$ contains a compact
totally geodesic submanifold $S$ (called a soul of $M$) such that
$M$ is diffeomorphic to the normal bundle of $S$.
\end{theorem}

When $sec(M)>0$, Gromoll and Meyer (\cite{[GMe]}) earlier showed
that a soul is a point, and thus $M$ is diffeomorphic to $\mathbb
R^n$. Cheeger and Gromoll proposed  the following so called Soul
Conjecture: If a complete noncompact nonnegatively curved Riemannian
manifold has strictly positive sectional curvature around a point,
then a soul is a point.

In 1994, Perelman (\cite{[Per3]}) proved the following theorem which
implies the Soul Conjecture:

\begin{theorem}
Let $M$ be a complete noncompact  Riemannian $n$-manifold with
$sec(M)\geqslant0$, and let  $S$ be a soul. If $P:M \to S$ is a
distance nonincreasing map, then the following properties hold:
\newline (0.2.1)  For any
$x\in S$ and any unit vector $v$ at $x$ normal to $S$,
$P(\exp_{x}(tv))=x, \text{ for all } t\geqslant0.$
\newline (0.2.2) Let $\gamma:[0,l]\to S$ be a geodesic, and let $V(s)$
denote the parallel vector field along $\gamma(s)$ with $V(0)=v$.
Then $\sigma_{s} (t)=\exp_{\gamma(s)}(tV(s))$ are geodesics filling
a flat totally geodesic strip $(t\geqslant0)$. If $\gamma([0,l])$ is
minimal, then $\sigma_{s}(t_0)|_{[0,l]}$ are minimal for any fixed
$t_0$.
\newline (0.2.3) $P$ is a  $C^1$-Riemannian submersion such that
the eigenvalue of the second fundamental form of $P$-fibers are
bounded above by the inverse of the injectivity radius of $S$ (in
the barrier sense).
\end{theorem}

Note that (0.2.2) implies that if $S$ is not a point, then any point
in $M$ is on some flat totally geodesic strip, and thus the Soul
Conjecture.

Note that it was shown independently by Cao-Shaw (\cite{[CS]}) and
Wilking (\cite{[Wi]}) that $P\ :M\to S$ is actually smooth (cf.
\cite{[Gu]}).

In this paper, we are concerned with analogue of Theorem 0.2 in
Alexandrov geometry. An Alexandrov space is a complete length space
on which Toponogov's triangle comparison holds. The study of
Alexandrov spaces was initiated by Burago-Gromov-Perelman
(\cite{[BGP]}), partially motivated by the fact that the
Gromov-Hausdorff limit of a sequence of Riemannian manifolds with
sectional curvature bounded from below uniformly is an Alexandrov
space which  in general may have both geometric and topological
singularities.

In [24, 6.3], Perelman extended Theorem 0.1 to Alexandrov spaces:

\begin{theorem}
Let $A$  be a complete noncompact nonnegatively curved Alexandrov
space. Then there is a compact convex subset $S$ without boundary
and a distance nonincreasing deformation retraction $\pi:A\to S$.
\end{theorem}
The map $\pi$ is called the Sharafutdinov retraction of $A$. In
further study following Theorem 0.3, there are two basic questions:
\begin{open problem}

(0.4.1) (\cite{[Per3]}) 
Soul Conjecture: If a complete noncompact nonnegatively curved
Alexandrov space has strictly positive curvature around a point,
then a soul is a point.

\noindent (0.4.2) ([1, 18.5]) Is the Sharafutdinov retraction $\pi$
a submetry? (A submetry is a map which preserves all $r$-balls, and
thus  submetry is a metric analogue of Riemannian submersion.)
\end{open problem}

An affirmative answer to (0.4.2) will easily imply (0.4.1) (see the
proof of Corollary 0.6), but the converse may not be true.

(In the following context, when we say a minimal geodesic from one
point to a subset, we always mean one whose length realizes the
distance from the point to the subset.)

In [42, 2.1], Yamaguchi partially generalized (0.2.2):

\begin{theorem}
Let $C$ be a convex closed subset in $A\in \Alex^n(0)$, with
boundary $\partial C\neq \emptyset$, let $f=\dist_{\partial C}$ and
let $\gamma(t)\subset C$ ($t\in [0,b]$) be a minimal geodesic with
$\gamma(0)=p,\ \gamma(b)=q$, such that $f(\gamma(t))=const$. Then
for any minimal geodesic $\gamma_0$ from $p$ to $\partial C$, with
$\angle(\gamma_0^+(0),\gamma^+(0))=\frac\pi2$, there is a minimal
geodesic $\gamma_1$ from $q$ to $\partial C$, such that $\{\gamma,
\gamma_0,\gamma_1\}$  bounds a flat totally geodesic rectangle.
\end{theorem}

As seen earlier that in Riemannian case, (0.2.2) implies Soul
Conjecture. In comparison, a gap between Theorem 0.5 and (0.4.1) is
that for a complete noncompact nonnegatively curved Alexandrov
space, there may be points where no flat totally geodesic rectangle
obtained in Theorem 0.5 passing through. However in the case that
$\text{codim}(S)=1$ (cf. \cite{[SY]}), Theorem 0.5 implies an
affirmative answer to (0.4.2).




In this paper, we will investigate the structure of a complete
noncompact nonnegatively curved Alexandrov space which is
topologically nice and whose a soul has codimension $2$. A point in
an Alexandrov space is called 
topologically nice if the iterated spaces of directions are all
homeomorphic to spheres. An Alexandrov space is called topologically
nice if all points on it are topologically nice. The limit space of
a sequence of noncollapsed Riemannian manifolds with sectional
curvature bounded from below uniformly is topologically nice.

We now begin to state the main results in this paper.

\vspace{0.3cm}

\noindent {\bf Theorem\hspace{0.1cm}A.} {\emph{ Let $A$ be a
complete noncompact nonnegatively curved Alexandrov space, and let
$\pi:A\to S$  be the Sharafutdinov retraction. Suppose that  $A$ is
topologically nice and that $S$  is of $codimension\ 2$. Then $\pi $
is a submetry. }}

\vspace{0.3cm}

We point out that the regularity assumption in Theorem A is used to
classify the space of directions of points on $S$, which is crucial
in the proof of Theorem A. Using Theorem A, one easily gets an
affirmative answer to Open Problem (0.4.1) in the following
$4$-dimensional topological manifold case.
\begin{corollary}
Let $A$ be a complete noncompact nonnegatively curved
$4$-dimensional  Alexandrov space. Suppose $A$ is a topological
manifold. If $A$ has positive curvature around a point, then a soul
is a point.
\end{corollary}
We now explain the main ideas in the proof of Theorem A. We may
assume that $A$ is simply connected. (If $A$ is not simply
connected, one can pass to the universal cover, see Lemma 1.5.)

For $p\in S$, and $v\in\Uparrow_p^{\partial \Omega_c}$ (all
directions at $p$ of minimizing geodesics from $p$ to $\partial
\Omega_c$),  where $c$ is a fixed noncritical value of the Busemann
function $f$ (defined in Section 1.1) and $\partial
\Omega_c=f^{-1}(c)$, there is always a ray $\sigma$ at $p$ such that
$\sigma^+(0)=v$. We call such ray a special normal ray to $S$. Let
$\mathscr{F}\subseteq A$ be the union of points on all such rays

\vspace{0.3cm}
\begin{center}
$\mathscr{F}=\{x\in A\ |\ x\in \sigma: a\   ray\ with\
\sigma(0)=p\in S,\sigma^+(0)=v\in\Uparrow_p^{\partial \Omega_c}\}.$
\end{center}

\vspace{0.3cm}

Observe that in the special case $\mathscr{F}=A$, Theorem A follows
easily from Theorem 0.5 (see the proof following Lemma 1.1.).

If $\mathscr{F}\neq A$, we set

\vspace{0.3cm}\begin{center}$F_v=\cup\ \{x\ |\ x\in$ flat totally
geodesic strips in $A$ spanned by $\sigma$ and all minimal geodesics
in $S$ from $p$ to all the points in $S\}$.
\end{center}

\vspace{0.3cm}We have the following
\begin{key lemma}
Let the assumptions be as in Theorem A, assume that $A$ is simply
connected. If $\mathscr{F}\neq A$, then for $v\in
\Ua_p^{\partial\Omega_c}$, $F_v$ with the restricted metric
isometrically splits, i.e., $F_v\overset{\rm isom}{\cong}S\times
\mathbb R^1_+$.
\end{key lemma}
Our proof of Key Lemma 0.7 relies on a property of space of
directions on $S$ (Proposition 2.1; note that if we can solve
Conjecture 4.5 completely for the case in Theorem A, then the proof
of Theorem A is a little simpler), where the regularity conditions
are required. Assuming Key Lemma 0.7, we can choose $F=\cup_{1\leq
i\leq l}F_{v_i},\ l\leq 3$ such that the distance function from $F$,
$\dist_F$, is concave in $A\backslash F$ (see Lemma 2.8). For any
given point $x\in A\backslash F$, let $\hat x\in (S,a)\subset
F_u\subset F$ such that $|x\hat x|=|x F|$. When $a\neq 0$, using the
concavity of $\dist_F$, we can construct a ``gradient flow" of
$\dist_F$ from $(S,a)$ passing $x$, denoted by $\Psi^t_a$, which is
distance nonincreasing (cf. \cite{[Pet2]}).

Consider the composition $i\circ\pi\circ \Psi_a^t:\,(S,a)\to (S,a)$,
(where $i:S\to(S,a)$ is the natural isometry,) which is distance
nonincreasing and a deformation, thus onto since $t=0$ is onto. A
standard argument shows that 
$\pi|_{\Psi_a^{|x\hat x|}((S,a))}$ is an isometry. We denote
$\Psi_a^{|x\hat x|}((S,a))$ by $S_x$. When $a=0$, we use a limit
argument (see 2.4) to get a $S_x\ni x$, such that $\pi|_{S_x}$ is an
isometry.

With the above preparations, we are ready to explain that the
Sharafutdinov retraction $\pi:\, A\to S$ is a submetry. First $\pi$
is distance nonincreasing (Theorem 0.3). For any $\bar y\in S$, it
suffices to find $y\in A$ such that $|xy|=|\pi(x)\bar y|$ and
$\pi(y)=\bar y$. Now it is clear that $y=S_x\cap \pi^{-1}(\bar y)$
satisfies the desired condition.

Our argument can be viewed as a generalization of \cite{[SY]}, where
noncompact nonnegatively curved Alexandrov spaces with souls of
codimension 1 are classified.

We organize the rest of the paper as follows:

In Section 1, we will collect some basic notions and properties
which will be used throughout the paper.

In Section 2, we will prove Theorem A by assuming Proposition 2.1.

In Section 3, some applications are proved there.

In Section 4, we will prove some structural results for spaces of
directions of points on $S$ and verify Proposition 2.1 at the end.

\section{Preliminaries}
We start this section with fixing some  notations:\\
$\dist_x(y)=|xy|$:\ \ \ \  the distance between points $x$, $y\in A$\\
\noindent $\Alex^n(\kappa)$:\ \ \ \   the class of complete
n-dimensional Alexandrov spaces with curvature $\geq \kappa$\\
$\partial A$:\ \ \ \ the boundary of A, $A\in \Alex^m(\kappa)$\\
\noindent $S^n(\kappa)$: \ \ \  the $n$-space form of curvature
$\kappa$\\
\noindent $B(p,r)=\{x\ |\ |xp|\leq r\}$\\
\noindent $S(p,r)=\{x\ |\ |xp|= r\}$\\
$Fr(C)$:\ \ \ \ the union of points whose any neighborhood
contains points in $C$ and in the complement of $C$\\
$[pq]$:\ \ \ \ \  a minimal geodesic from $p$ to $q$, $p,q\in A$\\
$\uparrow_p^q$:\ \ \ \  a direction at $p$ of a minimizing geodesic
  from $p$ to $q$\\
$\Uparrow_p^q$:\ \ \ \  the set of all directions at $p$ of  minimizing geodesics from $p$ to $q$\\
$\angle(\uparrow_y^x,\uparrow_y^z)$:\ \ \ \
the angle between  $\uparrow_y^x$ and $\uparrow_y^z$\\
$\tilde\angle_k(x,y,z)$:\ \ \ \  the corresponding comparison angle
  on space form $S^k_2$\\
$[CD]=\{x\ | \ x\in [cd],\ c\in C,\ d\in D\}$, where $C,\ D$ are
 subsets of $A$\\
 $C^{\perp}=\{v\in A\ |\ |vC|=\frac\pi2\}$, where $C\subset
 A$\\
 Flat totally geodesic strip $P$ in $X$, $X\in \Alex$:\ \ \ \ 
$P$ is the image of an isometric embedding from $\{(x,y)\in\mathbb
  R^2\ |\ 0\leq x\leq 1,y\geq 0\}$ with the standard flat metric to $X$\\


For basic notions related to  Alexandrov spaces, we refer to
\cite{[BBI]},
\cite{[BGP]}, \cite{[Pet1]}, \cite{[Pl2]} and \cite{[Shi]}. 

In the following, we shall briefly recall the construction of souls
using Busemann function and the construction of the Sharafutdinov
retractions. Then we shall establish some properties which will be
used in our proof or which may not be found in literature.

\subsection{Souls and Sharafutdinov retractions}
\ \\

Throughout this paper, we say a subset $C$ convex, if for any $p$,
$q$ in $C$, $C$ contains at least one minimal geodesic joining
$p,q$.

Let $A\in \rm {Alex}^n(0)$ be noncompact, and let $p\in A$. The
Busemann function at $p$ is defined by
\[
f(x)=\lim_{t\to \infty}(|x,S(p,t)|-t),
\]
and $f$ is a proper concave function with definite maximum
$a_0=\max\limits_{x\in A}\{f(x)\}$. Then $C_0=f^{-1}(a_0)$ satisfies
that for any two points, all minimal geodesics joining them are
contained in $C_0$ (and thus $C_0$ is convex). If $\partial
C_0=\emptyset$, then $C_0=S$, a soul of $A$. Otherwise, the distance
function, $f_1=\rm{dist}_{\partial C_0}: C_0\to \mathbb R^1$, is
again concave. Let $a_1=\max\limits_{x\in C_0}\{f_1(x)\}$, and let
$C_1=f_1^{-1}(a_1)$. Repeating the above process for $C_1$, and
after a finite number of steps we obtain $C_k=S$, a convex subset
without boundary.

Next we will recall the construction of a distance nonincreasing
deformation retraction from $A$ to $S$,  the so-called Sharafutdinov
retraction.

Let $\nabla_q f$ denote the gradient of $f$ at $q$. Since $f$ is
concave, there are $f$-gradient curves. We reparameterize gradient
curves so that a new curve $\alpha(t)\subset (A-C_0)$ satisfies that
$\alpha(0)=x$ and $\alpha^+(t)=\dfrac{\nabla_{\alpha(t)} f}{|\nabla
_{\alpha(t)}f|^2}$. Let $\beta(t)\subset (A-C_0)$ be the
reparametrization of the gradient curve with $\beta (0)=y$. Without
loss of generality, we may assume that $f(x)\leq f(y)$ and
$f(\alpha(t_0))=f(y)$. By a direct computation, we get
\begin{center}
$|\alpha(t)\beta(0)|^+(t)=-\langle \alpha^+(t), \Uparrow
_{\alpha(t)}^{\beta(0)}\rangle\leq 0,\ t\leq t_0.$
\end{center}
Hence $$|\alpha(t_0)\beta(0)|\leq |\alpha(0)\beta(0)|$$ and
\begin{center}
$|\alpha(t+t_0)\beta(t)|^+(t)=-\langle \alpha^+(t+t_0),\Uparrow
_{\alpha(t+t_0)}^{\beta(t)}\rangle-\langle
\beta^+(t),\Uparrow_{\beta(t)}^{\alpha(t+t_0)}\rangle\leq 0$.
\end{center}
Therefore $|\alpha(t+t_0)\beta(t)|\leq|\alpha(0)\beta(0)|$. 
From this we can get that $\alpha(t)$ can be uniquely extended to
include the points on $C_0$, denoted by $\bar \alpha (t)$. Define a
map, $\pi_0: \ A\to C_0$, by $\pi_0(x)=\bar\alpha(a_0-f(x))$, with
$x=\bar\alpha(0)$. We have showed that $\pi_0$ is distance
nonincreasing.


If $\partial C_0\neq \emptyset$, repeating the above, we obtain that
$\pi_1:C_0\to C_{1}$ is distance nonincreasing.  Eventually, we will
get the Sharafutdinov retraction $\pi=\pi_{k}\circ\cdots\circ
\pi_0$.

\subsection{Flat totally geodesic strips}
\ \\

The goal of this subsection is to give the following unbounded
version of Theorem 0.5, which is known to experts (\cite{[Ya]}).
Since we can not find a complete proof in literature, for the
convenience of readers, we include a proof here.

A useful alternative expression of $f$ is: for any
$c<a_0=\max\limits_{x\in A}\{f(x)\}$,  for $x\in
\Omega_c=f^{-1}([c,a_0])$, $f(x)=|x\partial \Omega_c|+c$ ([9,
Proposition A.1 (5)]).
\begin{lemma}
Let $A\in\Alex^n(0)$ be noncompact, and let $f$ be a  Busemann
function. Then the following properties hold:\newline \noindent
(1.1.1) For $ p\in S$, let $q\in
\partial\Omega_{f(q)}$ such that $\left|pq\right |=\left
|p\partial \Omega_{f(q)}\right |$. Then $\left[pq\right ]$ can be
extended to a ray $\gamma$, with $\gamma (0)=p$ and
$\left|p\gamma(t)\right |=\left|p\partial
\Omega_{f(\gamma(t))}\right |$, for any $t\geq 0$.\newline\noindent
(1.1.2) For $p\neq r\in S$,  there exits  a ray $\sigma$  with
$\sigma (0)=r$ and $\left|r\sigma(t)\right |=\left|r\partial
\Omega_{f(\sigma(t))}\right |$, for any $t\geq 0$, and
$\{\gamma,[pr],\sigma\}$ bounds a  flat totally geodesic strip.
\end{lemma}
\begin{proof}[Proof of Theorem A for the case that  $\mathscr{F}=A$]
For any $x\in A$, we have that $x\in\gamma$: a special normal ray
from $\bar x$.
Hence $\pi(x)=\bar x$ (see Lemma (1.1.1). For any $\bar y\in S$,
by Lemma (1.1.2), there is a flat totally geodesic strip determined
by $\{\gamma,[\bar x\bar y]\}$; in which we can find $y\in
\pi^{-1}(\bar y)$, such that $|\bar x\bar y|= |xy|$.
\end{proof}

In the proof of Lemma 1.1 we will use the following lemma.
\begin{lemma}[{\rm{[42, 2.5]}}]
Let $\Sigma\in \Alex^n(1)$, and let $C\subset \Sigma$ be a locally
convex closed subset without boundary with positive dimension. If
$v\in \Sigma$ such that $|vC|\geq\frac{\pi}{2}$, then
$|v\xi|=\frac{\pi}{2}$, for any $\xi\in C$.
\end{lemma}

We emphasize that Lemma  1.2 will be frequently used throughout the
paper.

Let $X$ be an Alexandrov space. For $p\in X$, let $T_pX$ (or $T_p$)
denote the tangent cone of $X$ at $p$, and let $\Sigma_pX$ (or
$\Sigma_p$) denote the space of directions of $X$ at $p$.
\begin{proof}[Proof of Lemma 1.1]
(1.1.1): Let $q_1\in A$ such that $f(q_1)<f(q)$ and
$\left|qq_1\right |=\left|q\partial \Omega_{f(q_1)}\right |$. Then
$f(p)-f(q_1)=\left|p\partial \Omega_{f(q)}\right |+\left|q\partial
\Omega_{f(q_1)}\right |\geq \left|pq\right |+\left|qq_1\right
|\geq\left|pq_1\right |\geq\left|p\partial \Omega_{f(q_1)}\right
|=f(p)-f(q_1)$. Thus 
$[pq]\cup [qq_1]$ is a minimal geodesic with the desired property.
Iterating this process, one can get the desired ray $\gamma(t)$.

(1.1.2): Note that $\left|pq\right |=\left |p\partial
\Omega_{f(q)}\right |$ implies that $|pq|=|qS|$. Then by the first
variation formula, $|\ua_p^q\Sigma_pS|\geq\frac\pi2$, and by Lemma
1.2, $|\ua_p^qv|=\frac\pi2$, for any $v\in\Sigma_pS$. Thus by
Theorem 0.5, for $t_1>0$, there exists a flat totally geodesic
rectangle $P_1$ with two of the edges $[p\gamma(t_1)]$ and $[pr]$.
Hence there is a corresponding point $r_1\in\partial \Omega_{f\circ
\gamma(t_1)}$ such that $|rr_1|=|r\partial \Omega_{f\circ
\gamma(t_1)}|$. By applying Theorem 0.5, we get another flat totally
geodesic rectangle $P_2$ with two of the edges
$[\gamma(t_1)\gamma(t_2)]$ and $[\gamma(t_1)r_1]$, there is a
corresponding point $r_2\in\partial \Omega_{f\circ \gamma(t_2)}$,
such that $|r_1r_2|=|r_1\partial \Omega_{f\circ \gamma(t_2)}|$. Next
we will show that $P_1\cup P_2$ is a flat totally geodesic
rectangle. There is a canonical map, $g:R=\{(x,y)\in \mathbb R^2\ |\
0\leq x\leq |pr|, 0\leq y\leq t_2\}\to P_1\cup P_2$, with
$g((0,0))=p$, and $g(\{x=0,0\leq y\leq t_2\})=\gamma$. In order to
show that $g$ is an isometry, it suffices to show that
$|g(z_1)g(z_2)|=|z_1z_2|$, for any $z_1,z_2\in\partial R$. First one
can easily see that $g(\{x,0\leq y\leq t_2\})$ are all minimal
geodesics, i.e., vertical direction are all isometry. The left cases
are similar. We just show the case of $z_1=(0,t_2)$ and
$z_2=(|pr|,0)$. For $\gamma(t),0\leq t\leq t_2$ and $[pr]$, we can
apply Theorem 0.5 to get another flat totally geodesic rectangle,
thus we get that $|\gamma(t_2)r|=\sqrt {t_2^2+|pr|^2}$, then
$|g(z_1)g(z_2)|=|z_1z_2|$ follows.

Let $P=\cup_{i\geq 0} P_i$. It follows that $P$ is a flat totally
geodesic strip and there is a corresponding geodesic ray from $r$
with the desired property.
\end{proof}
Note that flat totally geodesic strip in Lemma (1.1.2) may not be
unique, see example [42, 14.8].
\begin{remark}
Inspecting the proof of Lemma 1.2, one can see that when $\partial
C\neq \emptyset$ the following holds: Let $x\in C$ be a point such
that $|vC|=|vx|$. If $x\not\in\partial C$, then
$|v\xi|=\frac{\pi}{2}$, for any $ \xi\in C$. We will use this
observation in Section 4.
\end{remark}

\subsection{A reduction}
\ \\

The goal here is to reduce the proof of Theorem A to the simply
connected case.

Let $A\in \Alex^n(\kappa)$. Recall that $p\in A$ is  topologically
regular, if $\Sigma_pA$ is homeomorphic to a sphere. A topologically
regular point has a neighborhood homeomorphic to a Euclidean ball.
$A$ is called topologically regular, if all points are topologically
regular, and thus $A$ is a topological manifold.

A point $p\in A$ is called topologically nice, if the iterated
spaces of directions, i.e., $\Sigma_pA,\ \Sigma_{v_1}\Sigma_pA,\
...$ are all homeomorphic to spheres. $A$ is called topologically
nice, if all points are topologically nice.
Topologically nice implies topologically regular, but the converse
may not be true.
\begin{example}[{\rm{\cite{[Ka2]}}}]
Let $\Sigma^3$ be the Poincar\'e homology 3-sphere with constant
curvature $1$. 
Then the three-fold spherical suspension of $\Sigma^3$,
$S^3(\Sigma^3)\in\Alex^6(1)$, 
is topologically regular but not topologically nice.
\end{example}
\begin{lemma}
If Theorem A holds for simply connected Alexandrov spaces, then it
holds for any non-simply connected Alexandrov spaces.
\end{lemma}
\begin{proof}
Let $\tilde S$ be the universal cover of $S$. Denote the covering
map by $\phi$. Let $\phi^*(A)=\{(p,e)\in \tilde S\times A\ |\
\phi(p)=\pi(e)\}\subset \tilde S\times A$, with the induced
topology. Then by a standard argument, we obtain that $\phi^*(A)$ is
the universal cover of $A$, and $\tilde \phi:\phi^*(A)\to A$,
defined by $\tilde \phi ((p,e))=e$, is the covering map.

Endow $\phi^*(A)$ with the induced metric, denoted by $\tilde A$.
Then $\tilde \phi:\tilde A\to A$ is a local isometry, and $\tilde
\pi:\tilde A\to\tilde S$, with $\tilde\pi((p,e))=p$, is locally
1-Lipschitz.

First we assume that $C_0=S$.
\begin{sub lemma}
Let $\tilde \Omega _c=\{(p,e)\in \tilde S\times \Omega_c\ |\
\phi(p)=\pi(e)\}$. Then $\tilde\Omega_c\subset\tilde A$ is convex.
\end{sub lemma}
\begin{proof}
For $ x,y\in\tilde\Omega _c$, if $[xy]\subsetneqq \tilde \Omega _c$,
by the construction of $\tilde A$, there exists a curve in
$\tilde\Omega _c$ with length $\leq |xy|$, a contradiction.
Thus we get the sublemma.
\end{proof}
Let $\tilde f=\dist_{\partial\tilde \Omega_c}$. By the property of
covering space, $\tilde f((p,e))=|(p,e),\partial\tilde \Omega
_c|=|e,\partial\Omega_c|=f(e)-c$. It follows that
$\partial\tilde\Omega_c$ are level sets of $\tilde f$. And by the
local isometry, $|\nabla \tilde f_{(p,e)}|=|\nabla f_e|$. Hence
$\alpha (t)$ is an $f$-gradient curve if and only if $(p,\alpha(t))$
is a $\tilde f$-gradient curve.

If $\tilde S$ is compact, by the assumption of the lemma, we can see
that $\tilde \pi$ is a submetry (since $\tilde A$ is topologically
nice). Hence $\pi$ is also a submetry.

If $\tilde S$ is not compact, by the splitting theorem \cite{[Mi]},
there is an isometric splitting $\tilde S=\mathbb R^k\times S_0$,
where $S_0$ is simply connected and compact, exactly as the proof of
Riemannian case. It follows that $\tilde A=\mathbb R^k\times A_0$
and $\partial\tilde\Omega _c=\mathbb R^k\times\partial \Omega_c'$.
We claim that $\nabla_{(x,x_0)}\dist_{\tilde\po}\in T_{x_0}A_0$, for
any $(x,x_0)\in\tilde A$. Hence $\tilde\pi=(id, \pi_0)$.

Since $\tilde A=\mathbb R^k\times A_0$ is topologically nice, we
have that $A_0$ is topologically nice (see Remark 1.7), as can be
seen in the proof of [32, Theorem D]. Thus by the assumption of the
lemma, we know that $\pi_0:\ A_0\to S_0$ is a submetry. It follows
that $\tilde\pi$ is a submetry.

Finally, we will verify the claim: for any $x_0\in A_0$ and
$y=(y_1,x_0)\in \mt R^k\times A_0$, we have that $\Sigma_y\tilde
A=S^{k}(\Sigma)$, where $\Sigma=\Sigma_{x_0}A_0$.
Let $v$ be a point such that $|v\Ua_y^{\partial
\tilde\Omega_c}|=\max\limits_{w\in\Sigma_y\tilde A}
\{|\Ua_y^{\partial \tilde\Omega_c}w|\}$. Since
$\partial\tilde\Omega_{\bar c}=\mathbb R^k\times\partial
\Omega_{\bar c}'$, for any $\bar c<a_0$, we have that
$\Ua_y^{\partial \tilde\Omega_c}\in \Sigma$.  It follows that
$v\subset \Sigma$. By the definition of gradient the claim follows.

If $C_0\neq S$, consider $\dist _{\partial\Omega_{a_0}}$ instead of
$f$, we can get the same conclusion.
\end{proof}

\begin{remark}
If $A$ is only topologically regular, then $A_0$ may not be a
topological manifold, even $\tilde A=\mathbb R^k\times A_0$ is a
topological manifold (\cite{[Ka2]}).
\end{remark}

\section{Proof of Theorem A} 

In our proof of Theorem A, 
the following structural results on spaces of directions of points
on soul plays a curial role.

\begin{proposition}
Let the assumptions  be as in Theorem A. For $p\in S$, let $
\Sigma_0^p=\Sigma_p S$, and let $\Sigma_1^p=\{ v\in \Sigma_p A
\,\mid\,|v\Sigma_pS|\geq\frac\pi 2\}$. 
Then
\newline\noindent (2.1.1) $\Sigma_0^p$ is 
homeomorphic to a sphere. \newline\noindent (2.1.2) $\Sigma_1^p$ is
convex\footnote{See Remark 4.3 (2), (1).}  and isometric to one of
the following:

$S^1(r)\ with\ r\leq 1$, $[ab]$, $\{v\}$,   $\{v_1,v_2\}\ with\
|v_1v_2|=\pi$.
\newline\noindent (2.1.3) With the restricted metric,
$[\Sigma_0^p\Sigma_1^p]=\Sigma_0^p*\Sigma_1^p$ with the standard
join metric\footnote{See Definition 4.1}.
\newline\noindent (2.1.4) When $\Sigma_1^p=[ab]$, if there
exists a subset $E$ of $\Sigma_1^p$ such that
$B(E,\frac\pi2)=\Sigma_pA$, then $a,b\in E$. If $|ab|=\pi$ with $o$
the middle point of $[ab]$, then for any $x\in \Sigma_pA$ such that
$|ox|\leq\frac\pi2$, $x\in[\Sigma_0^p\Sigma_1^p]$.
\end{proposition}
Because the proof of Proposition 2.1 is technical and long, we will
postpone the proof to the next section. Below we shall prove Theorem
A by assuming Proposition 2.1.
\subsection{Proof of Key Lemma 0.7}
\ \ \ \\

Recall that
\begin{center} $F_v=\cup\ \{x\ |\ x\in$ flat totally
geodesic strips in $A$ spanned by $\sigma$ and all minimal geodesics
in $S$ from $p$ to all the points in $S\}$.\end{center}

\noindent where $p\in S$, $v\in \Uparrow_p^{\po}\subset\Sigma_1^p$
($c<\max {f}$),
$\sigma(t)\subset A$ is a ray with $\sigma(0)=p$ and $\sigma^+(0)=v$
(see Lemma (1.1.1)), and if $\gamma\subset S$ is a minimal geodesic
from $p$ to $q\in S$, then $\{\sigma,\gamma\}$ determines a unique
flat totally geodesic strip (see Lemma (1.1.2)); the uniqueness
follows from Proposition (2.1.3), otherwise, will violate the join
of $[\Sigma_0^p\Sigma_1^p]$.

Let $\alpha$ be the other ray from $q$ which bounds the flat totally
geodesic strip, with $\bar v=\alpha^+(0)\in\Ua_q^{\po}$. We then
define a map
$\phi_{[pq]}:\ \Ua_p^{\po}\to \Ua_q^{\po}$, by $\phi_{[pq]}(v)=\bar
v$. Note that $\phi_{[pq]}$ may depend on the choice of $[pq]$.

Observe that for all $q\in S$, $\phi_{[pq]}$ is independent of
$[pq]$ if and only if $F_v$ is isometric to $S\times \mathbb R^1_+$,
i.e., Key Lemma 0.7 holds.

To prove the independency, we will first show that $\phi_{[pq]}$ is
an isometry. (We point out that the method we used in the proof of
Lemma 2.2 below was previously used in \cite{[Pet1]}.)
\begin{lemma}
Let the assumptions be as in Theorem A. For every $x\in S$ and every
$y\in S$, $\phi_{[xy]}:\ \Ua_x^{\po}\to \Ua_y^{\po}$, is an
isometry, for any minimal geodesic $[xy]$.
\end{lemma}
\begin{proof}
For $u,v\in\Ua_x^{\po}$, and $\varepsilon>0$, let $\bar
v_{\varepsilon}=|y,\exp_x(\varepsilon v)|\ua_y^{\exp_x(\varepsilon
v)}$, where $\exp_x:T_x\to A$ is the usual exponential map, and
$\ua_y^{\exp_x(\varepsilon v)}$ is the direction of the minimal
geodesic contained in the flat totally geodesic strip determined by
$\exp_x(t v),t\geq 0$ and $[xy]$. Similarly for $\bar
u_{\varepsilon}$, with $\varepsilon u$ instead of $\varepsilon v$ in
the definition of $\bar v_{\varepsilon}$.

Since by Proposition (2.1.3)
and by the property of flat totally geodesic strip, we have that
$|\bar v_{\varepsilon}^{\perp}\bar u_{\varepsilon}^{\perp}|_{T_y}=
|\bar v_{\varepsilon}\bar u_{\varepsilon}|_{T_y}\geq
|\exp_x(\varepsilon v)\exp_x(\varepsilon u)|=\varepsilon|
uv|_{T_x}+o(\varepsilon)$, where $\perp $ is the orthogonal
projection to $C(\Sigma_1^y)$ (Euclidean cone over $\Sigma_1^y$).

Let $\bar u=\lim\limits_{\varepsilon\to0}\frac{1}{\varepsilon}\bar
u_{\varepsilon}^{\perp}=\phi_{[xy]}(u)$ and $\bar
v=\lim\limits_{\varepsilon\to 0}\frac{1}{\varepsilon}\bar
v_{\varepsilon}^{\perp}=\phi_{[xy]}(v)$. Then $|\bar u\bar
v|_{T_y}\geq |uv|_{T_x}$, and $|\bar u|=|u|,|\bar v|=|v|$. Thus we
get that
$|\phi_{[xy]}(u)\phi_{[xy]}(v)|_{\Sigma_y}\geq|uv|_{\Sigma_x}$.

Similarly, the opposite inequality holds. Hence
$|\phi_{[xy]}(u)\phi_{[xy]}(v)|_{\Sigma_y}=|uv|_{\Sigma_x}$.
\end{proof}
\begin{remark}
It seems that Lemma 2.2 can be strengthened to 
that the isometric class of $\Sigma_1^p$ is independent of $p$.
\end{remark}
\begin{proof}[Proof of Key Lemma 0.7]
Define a map, $\psi: F_v\to S\times \mathbb R^1_+$,
$\psi(\exp_q(t\phi_{[pq]}(v)))=(q,t)$. As pointed out earlier,
$\psi$ is an isometry if $\phi_{[pq]}(v)$ is independent of $q\in
S$. In view of the simply connectedness of $S$ (because $A$ is
simply connected), first we will show that  $F_v$ is a product
locally,
it suffices to show that locally $\phi_{[pq]}(v)$ is independent of
$q\in S$ (all $p$). Precisely, for $x \in S$, there exists
$\varepsilon>0$ ($\varepsilon$ depends on $x$), such that for any
$y$, $z\in B(x,\varepsilon)$, $g=\phi_{[zx]}\circ
\phi_{[yz]}\circ\phi_{[xy]}=id$.

If $\Sigma_1^x=\{v\}$, then $\Ua_x^{\po}=\Sigma_1^x$. Thus each
point in $S$ has just one special normal ray to $S$, clearly $g=id$.

For other cases we will argue by contradiction. Suppose that for a
sequence $ \varepsilon_i\to 0$, there exist $y_i,z_i\in
B(x,\varepsilon_i)$, $g_i=\phi_{[z_ix]}\circ
\phi_{[y_iz_i]}\circ\phi_{[xy_i]}\neq id$.

If $\Sigma_1^x=\{v_1,v_2\}$ with $|v_1v_2|=\pi$,  then
$g_i(v_1)=v_2$, and by the property of flat totally geodesic strips,
we have that $|\exp_x(v_1)\exp_x(g_i(v_1))|\leq (|xy|+|yz|+|zx|)\leq
4\varepsilon_i$. When $\varepsilon_i\to 0$, we get a contradiction,
since $|\exp_x(v_1)\exp_x(v_2)|>0$, or geodesic branches.

If $\Sigma_1^x=S^1(r)$ with $r\leq 1$, by Lemma 2.4 below, we have
that every $g_i$ is the restriction of an isometry, $\bar g_i:S^1\to
S^1$, which is a rotation or a reflection. By passing to a
subsequence, we can suppose that every $\bar g_i$ is a rotation or
every $\bar g_i$ is a reflection.

(a): Every $\bar g_i$ is a rotation.

For $v\in\Ua_x^{\po}$,
$|\exp_x(tv)\exp_x(g_i(tv))|\to 0$, for any $t\geq 0$,  as
$\varepsilon_i\to 0$, which can be seen in the above case. Hence
$|vg_i(v)|\to 0$, i.e. $|v\bar g_i(v)|\to 0$, as $\varepsilon_i\to
0$. Then by the closeness of $\Uparrow_x^{\po}$, we can get that
$\Uparrow_x^{\po}=S^1$. This is a contradiction, since
$\mathscr{F}\neq A$ and Lemma 2.2 imply that
$\Ua_q^{\partial\Omega_c}\neq S^1,$ for any $q\in S$.

(b): Every $\bar g_i$ is a reflection.

By passing to a subsequence, we can assume that $\bar g_i\to h$,
which is also a reflection. Observe that there is $v\in\Ua_x^{\po}$,
such that $v\neq h(v)$, or $g_i$ will be equal to $id$, a
contradiction.
Similarly, we have that $|\exp_x(tv)\exp_x(g_i(tv))|\to 0$, for any
$t\geq 0$, as $\varepsilon_i\to 0$. Hence $|vg_i(v)|\to 0$, i.e.
$|v\bar g_i(v)|\to 0$, as $\varepsilon_i\to 0$. 
Thus we have that $v=h(v)$, a contradiction.

If $\Sigma_1^x=[ab]$, likewise by Lemma 2.4, each $g_i$ just can be
the restriction of the reflection of $[ab]$. Similarly as above, we
can get the conclusion.

Then 
$F_v|_{B(x,\varepsilon)}$ is a product. (We call $B(x,\varepsilon)$
a local product neighborhood of $x$.)

Finally, we will show that $F_v$ is a product globally. For any $
q,r\in S$, and for three fixed geodesics $[xq],[qr]$ and $[rx]$, let
$\gamma=[xq]\cup [qr]\cup [rx]$. Since $\pi_1(S)=0$, $\gamma$ is
homotopic to a point. Let $H:[0,1]\times [0,1]\to S$ be a homotopy,
with $H(t,1)=\gamma(t),H(t,0)=x,H(0,s)=x$ and $H(1,s)=x$. Let
$\{s_0=0<s_1<\cdots <s_n=1\}$ and $\{t_0=0<t_1<\cdots<t_n=1\}$ be
two partitions of $[0,1]$, such that
$H([s_i,s_{i+1}]\times[t_j,t_{j+1}])\subset U_z$; a closed convex
neighborhood of some $z\in S$ which is contained in the local
product neighborhood of $z$. Let $\sigma_i=\cup_{j=0}^{n-1}
[H(s_i,t_j)H(s_i,t_{j+1})]$. There are corresponding $g_i$. We can
see that $g_i(w)$, for any $w\in \Uparrow_x^{\po}$, are the same for
any $i$. It follows that $g_1(w)=w$, i.e., $g_1=id$. Hence
$\phi_{[rx]}\circ \phi_{[qr]}\circ\phi_{[xq]}=id$. Thus we can get
that $\psi$ is an isometric map from $S\times \mt R_+^1$ to $F_v$.
\end{proof}
\begin{lemma}
Let $M\overset{\rm{homeo}}{\cong}S^1$ or an interval and with
intrinsic metric, let $N\subset M$ be a subset and let $g: N\to N$
be an isometry, where $N$ is with the restricted metric. Then $g$
can be extended to an isometry $\bar g:\ M\to M$.
\end{lemma}
\begin{proof}
We just show the case of $M\overset{\rm{homeo}}{\cong}S^1$,
similarly for an interval.

If there exist $v,w\in N$ such that $v,w$ are not antipodal, then
for any $u\in M$, $u$ is uniquely determined by $|uv|,|uw|$. Thus
$g$ is uniquely determined by $g(v),g(w)$. Hence $g$ can be extended
to $\bar g$, by $\bar g(u)=x$, where $x$ is the unique point such
that $|xg(v)|=|uv|$ and $|xg(w)|=|uw|$.

If not, then $N=\{v,w\}$ with $v,w$ antipodal. Clearly $g$ is
extendable.
\end{proof}
Next we will show
\begin{lemma}
For every $x\in F_u$ and every $y\in F_u$, $[xy]\subset F_u$.
\end{lemma}
In the proof of Lemma 2.5, we need the following lemma.
\begin{lemma}[{\rm{[14, 2.4(\romannumeral2')]}}]
Let $X\in \Alex^m(\kappa)$. For two minimal geodesics $[xz]$ and
$[xy]$, if $\angle(\uparrow_x^z,\uparrow_x^y)=\tilde \angle(z,x,y)$,
then there is a $[zy]$ such that $[xz],[xy],[zy]$ bound a totally
geodesic surface which is isometric to a geodesic triangle in
$S^2(\kappa)$.
\end{lemma}
\begin{proof}[Proof of Lemma 2.5]
Set $\pi(x)=\bar x$ and $\pi(y)=\bar y$. If $x, y\in S$, by the
construction of $S$, we have that $[xy]\subset S$.

For other cases we will argue by contradiction. Suppose that there
exist $x, y\in F$ such that there exists $[xy]$ which doesn't belong
to $F$. Then $[xy]^{\circ}\cap F=\emptyset$, where $[xy]^{\circ}$
denotes $[xy]-\{x,y\}$, or geodesic will branch.

If $x\in S,\, y\in (S,a)$, with $a\neq 0$, let $r\in F_u,r\neq y$ be
a point such that $\pi(r)=\bar y$ and $|r\pi(r)|>|y\bar y|$. Then
$\pi=\angle(\ua_y^x,\ua_y^{\bar y})+\angle(\ua_y^x,\ua_y^r)\geq
\tilde\angle(x,y,\bar y)+\tilde\angle(x,y,r)=\pi$, where
the last equality is from the construction of $F$. Hence
$\angle(\ua_y^x,\ua_y^{\bar y})=\tilde\angle( x,y,\bar y)$. It
follows from Lemma 2.6 that $\{x,y,\bar y\}$ bounds another flat
totally geodesic triangle, which contradicts to the structure of
$\Sigma_{\bar y}A$ (Proposition (2.1.3)), since $[x\bar y]\subset
S$.

If $x, \,y\in(S,a)$, with $a\neq 0$, then the same as above we have
that $\angle (\ua_x^y,\ua_x^{\bar x})=\tilde\angle (y,x,\bar
x)=\frac{\pi}{2}$. Therefore there exists a flat triangle bounded by
$y,x,\bar x$ for the given $[xy]$, with $[y\bar x]^{\circ}\nsubseteq
F$. By the above case, we get a contradiction.

If $x\in(S,a),\ y\in (S,b)$, with $a\neq 0$, $b\neq0$ and $a\neq b$,
without loss of generality, we can assume that $a<b$. Let $s=
[y\bar y]\cap (S,a)$. 
Similarly as the above two cases, we can also get a contradiction.
\end{proof}
As seen following Lemma 1.2, the remaining case in the proof of
Theorem A is that $\mathscr{F}\neq A$ and $\pi_1(A)=0$, which
implies that $\Ua_p^{\po}\neq S^1$.

\subsection{The concavity of $\dist_F$}
\ \\

As seen in the introduction, $F$ is the union of several $F_v $'s.
We point it out that the selection of these $F_v$'s is crucial for
the desired concavity of $\dist_F $; see following for details.

For $p\in S$, by the first variation formula for the Busemann
function, $d_pf(v)=-<\Ua_p^{\po},v>\leq 0$, $v\in\Sigma_p$, we see
that $\Ua_p^{\po}$ is $\frac{\pi}{2}$-dense in $\Sigma_p$ and thus
$\frac{\pi}{2}$-dense in $\Sigma_1^p$.

\begin{lemma}
There is $N'=\{v_i\}_{1\leq i\leq l}\subset \Ua_p^{\po},l\leq 3$,
such that $N'$ is $\frac{\pi}{2}$-dense in $\Sigma_1^p$, and
$\phi_{[pq]}(N')$ is also $\frac{\pi}{2}$-dense in $\Sigma_1^q$, for
any $q\in S$.
\end{lemma}
\begin{proof}
First for the selection of $N'$: if $\Sigma_1^p=\{v\}$, then
$\Ua_p^{\po}=\Sigma_1^p$.  Let $N'=\Uparrow_p^{\po}$.

If $\Sigma_1^p=\{v_1,v_2\}$ with $ |v_1v_2|=\pi$, then $v_1,v_2\in
\Ua_p^{\po}$. Indeed, since $ \Ua_p^{\po}\neq\emptyset$, one of
them, say $v_1$, must be in $ \Ua_p^{\po}$. Suppose that $v_2\not
\in \Ua_p^{\po}$, then $d_pf(v_2)=1>0$, a contradiction. Let
$N'=\Uparrow_p^{\po}$.


If $\Sigma_1^p=[ab]$,  since $B(\Ua_p^{\po},\frac\pi2)=\Sigma_pA$,
by Proposition (2.1.4), we have that $a,b\in \Ua_p^{\po}$. Let
$N'=\{a,b\}$. Then $N'$ is obviously a $\frac{\pi}{2}-$dense subset
of $\Sigma_1^p$.

If $\Sigma_1^p=S^1(r)$ with $r\leq 1$, we choose $v\in \Ua _p^{\po}$
arbitrarily, consider the antipodal point of $v$, denoted by $w$. If
$w\in\Ua_p^{\po}$, let $N'=\{v,w\}$. If $w\not \in\Ua_p^{\po}$, let
$v_1,v_2\in\Ua_p^{\po}$ be the farthest points to $w$ from both
sides respectively in $\Sigma_1^p$, it follows that $v,v_1,v_2$ ($v$
may be equal to $v_1$ or $v_2$) form a $\frac{\pi}{2}-$dense subset
of $\Sigma_1^p$, let $N'=\{v,v_1,v_2\}$.

For the second part of the lemma: if $\Sigma_1^p=\{v\}$ or
$\{v_1,v_2\}$, by Lemma 2.2 one can deduce that $N'$ has the desired
property.

If $\Sigma_1^p=[ab]$, suppose that there is $q\in S$ such that
$\phi_{[pq]}(N')$ is not a $\frac{\pi}{2}-$dense subset of
$\Sigma_1^q$. then there is $w\in\Ua_p^{\po}$ such that
$|\phi_{[pq]}(w)\phi_{[pq]}(a)|+|\phi_{[pq]}(w)\phi_{[pq]}(b)|>|ab|$,
a contradiction, since by Lemma 2.2, $\phi_{[pq]}$ is an isometry
when restricted to $\Ua_p^{\po}$.

If $\Sigma_1^p=S^1(r)$ with $r\leq 1$, then from the choosing
method, we have that for any $u\in\Ua_p^{\po}$, either
$|uv|+|uv_1|=|vv_1|$, or $|uv|+|uv_2|=|vv_2|$. Suppose that there is
$q\in S$ such that $\phi_{[pq]}(N')$ is not a $\frac{\pi}{2}-$dense
subset of $\Sigma_1^q$. Then there is $w\in\Ua_p^{\po}$ such that
$|\phi_{[pq]}(w)\phi_{[pq]}(v)|+|\phi_{[pq]}(w)\phi_{[pq]}(v_1)|>|vv_1|$,
and
$|\phi_{[pq]}(w)\phi_{[pq]}(v)|+|\phi_{[pq]}(w)\phi_{[pq]}(v_2)|>|vv_2|$,
a contradiction.

Thus we finish the proof of the lemma.
\end{proof}

\begin{lemma}
Let $F=\cup_{v\in N'} F_{v}$. The distance function, $\dist_F$, is
concave in $D=X\backslash F$.
\end{lemma}
\begin{remark}
Observe that if the boundary points of each component $D_i$ are
``true'' boundary points, i.e., which are not interior points in the
closure $\bar D_i$ and that $\bar D_i$ is convex, then it follows
that $\dist_F$ is concave in $D$. In our case, we show that even if
a component of $D$ may not be convex, $\dist_F$ is still concave.
For example: let $T=\{(x,y)\in\mathbb R^2\ |\ x\geq 0,y\geq 10
x\}\cup\{(x,y)\ |\ x\leq 0,y\geq x^2\}$, the metric product,
$A=Doub(T)\times S^2(1)\in\Alex ^4(0)$. Then $F=Doub(\{x=0,y\geq
0\})$ serves an example.
\end{remark}
In the proof of Lemma 2.8, we need the following lemma which is an
analogue to the totally geodesic property in Riemannian geometry.
(Lemma 2.10 below is from a helpful discussion with Shicheng Xu.)

\begin{lemma}
Let $X\in \Alex ^n(\kappa)$, and let $Y$ be a closed subset of $X$
such that for $x,y\in Y$, any $[xy]\subset Y$. Then for any $p\in Y$
and $q\in Y\backslash
\partial Y$, we have that $\nabla _q \dist_p\in T_qY$.
\end{lemma}

We don't know whether Lemma 2.10 is true for convex subset or not.
\begin{proof}[Proof of Lemma 2.10]
If $\nabla _q \dist_p=0$, nothing need to prove. Hence we can assume
$\nabla _q \dist_p\neq 0$. Since $\frac{\nabla _q\dist_p}{|\nabla
_q\dist_p|}=\{u\in\Sigma_qX\hspace{1mm}|\hspace{1mm}|u\Ua_q^p|=\max\limits_{v\in\Sigma_q}\{|v\Ua_q^p|\}
$, and by the condition of the lemma, we have that
$\Uparrow_q^p\subset \Sigma_qY$, it suffices to show that for
$V\subset \Sigma_qY$, let $ w\in \Sigma_qX$ be a point such that
$|Vw|=\max \{|V,\ |\}>\frac{\pi}{2}$, then we have that $w\in
\Sigma_qY$.

Argue by contradiction. Suppose that $w\not\in \Sigma_qY$. Choose
$w_0\in \Sigma_qY,\ \text{such that}$ $|ww_0|=|w\Sigma_qY|$. By
Lemma 1.2, $|\ua_{w_0}^{w}\bar v|=\frac\pi2$, for any $ \bar v \in
\Sigma_{w_0}\Sigma_qY$ and $|ww_0|\leq \frac{\pi}{2}$. Hence
$\tilde\angle (\tilde w,\tilde w_0,\tilde V)\leq \frac{\pi}{2}$. By
hinge comparison, we have that $|Vw_0|>\frac{\pi}{2}$. It follows
that $|Vw|<|Vw_0|$, a contradiction to the choice of $w$.
\end{proof}

Recall a standard fact in topology (cf. \cite{[Bro]}): If $X\subset
S^m $ is a closed $(m-1)$-topological manifold as a subspace, then
$S^m-X$ has two connected components, each having $X$ as its set
boundary. We say that $X$ separates $S^m$.

\begin{proof}[Proof of Lemma 2.8]
Given $q\in D$, let $\gamma(t)\subset D$ be a minimal geodesic with
$\gamma(0)=q$, and let $x\in F$ be a point such that $|qx|=|qF|$. By
a standard contradiction argument, one can get that for
$v=\ua_{x}^{\gamma(t)}\in \Sigma_{x}A$, when $t$ is small enough,
there is $\ua_{x}^q \ \text{such that}\ |v\ua_{x}^q|<\frac{\pi}{2}$.
(Note that for different $t$, $\ua_{x}^q$ may be different.)

From the proof of the concavity of distance function to the boundary
of an Alexandrov space (cf. [30, Theorem 3.3.1], [9, Lemma 3.1]),
one can deduce that if $F$ satisfies the following two conditions:

(\romannumeral1) there is $\bar w\in \Sigma_{x}F$ such that $|\ua
_{x}^q\bar w|=|\ua_{x}^qv|+|v\bar w|=\frac{\pi}{2}$,

(\romannumeral2) there is a radial curve, $\sigma:
[0,\varepsilon]\to F$, with $\sigma(0)=x$ and $\sigma^+(0)=\bar w$,
for some $\varepsilon>0$,

then $\dist_F$ is concave in $D$.


Thus it suffices to check that $F$ satisfies the two conditions.

For condition (\romannumeral1): if $x\in S$, by the first variation
formula, we have that $|\ua_{x}^q\Sigma_{x}F|\geq\frac\pi2$.
Specially, $|\ua_{x}^q\Sigma_{x}S|\geq\frac\pi2$, thus
$\ua_{x}^q\in\Sigma_1^x$.
Hence 
$\Sigma_1^x=S^1$ or $[ab]$ with $|ab|=\pi$ and $\uparrow_{x}^q$ the
middle point of $[ab]$. And by Lemma 2.7, we have that for any
$v\in\Sigma_{x},|v\phi_{[px]}(N')|\leq\frac{\pi}{2}$. Hence there
are two of $N'$ say $v_1,v_2$ ($v_1$ may be equal to $v_2$) such
that $|\phi_{[px]}(v_1)\uparrow_{x}^q|=\frac{\pi}{2}$,
$|\phi_{[px]}(v_2)\uparrow_{x}^q|=\frac{\pi}{2}$. Since
$|\ua_x^qv|<\frac\pi2$, by Proposition (2.1.4), we have that $v\in
[\Sigma_0^x\Sigma_1^x]$. Thus by Proposition (2.1.3), we have that
$v\in [w_1w_0]$, for some $w_i\in \Sigma_i^x$. We can suppose that
$w_1\in [\ua_{x}^q\phi_{[px]}(v_2)]$. Then $\triangle
(\ua_{x}^q,w_0,\phi_{[px]}(v_2))$ is isometric to a triangle with
three side lengths $\frac\pi 2$ on $S^2(1)$. Hence there is $\bar
w\in \Sigma_{x}F$ such that $|\ua _{x}^q\bar w|=|\ua_{x}^qv|+|v\bar
w|=\frac{\pi}{2}$.

If $x\bar\in S$,
$\Sigma_{x}F=S(\Sigma_{x}S)\overset{homeo}{\simeq}S^{n-2}$ is convex
in $\Sigma_{x}A$ and separates $\Sigma_xA$, and by the first
variation formula, $|\ua_{x}^q\Sigma_{x}F|\geq\frac\pi2$. Then by
Lemma 4.19 below, we have that there is $\bar w\in \Sigma_{x}F$ such
that $|\ua _{x}^q\bar w|=|\ua_{x}^qv|+|v\bar w|=\frac{\pi}{2}$.

For condition (\romannumeral2):
since $\bar w\in \Sigma_{x}F$, without loss of generality, we can
assume that $\bar w\in\Sigma_{x}F_{v_1}$.  Since
$(\Sigma_{x}F_{v_1})'$ is dense in $\Sigma_{x}F_{v_1}$, there are
$q_i\in F_{v_1}$, such that $q_i\to x$ and $\uparrow_{x}^{q_i}\to
\bar w$. Let $\sigma_i$ be the radial curve at $q_i$ with respect to
$x$. By [1, Chapter 15] or [27, 3.4], we know that if we can show
that
$\lim\limits_{i\to\infty}\sigma_i([0,\varepsilon])=\sigma([0,\varepsilon])\subset
F_{v_1}$ for some small $\varepsilon>0$, then $\sigma$ is the
desired  radial curve.

If $x\bar\in S$, by Lemma 2.10, we can get the desired radial curve.
If $x\in S$ and $\bar w\in\Sigma_xS$, similarly by Lemma 2.10, we
can get the desired radial curve in $S$. If  $x\in S$ and $\bar
w\not\in\Sigma_xS$, we can choose $q_i\in F_{v_1}-S$. We claim that
$\sigma_i([0,\infty))\subset F_{v_1}$. Therefore we can get the
desired radial curve.

Finally, we will verify the claim by showing that $\sigma_i$ are
more and more farther away from $\partial F_{v_1}$. The reason is
that $(\dist_S\circ
\sigma_i(t))^+=-\langle\uparrow_{\sigma_i(t)}^{\pi(\sigma_i(t))},
\sigma_i^+(t)\rangle=-\langle\uparrow_{\sigma_i(t)}^{\pi(\sigma_i(t))},
\frac{|x\sigma_i(t)|}{t}\nabla_{\sigma_i(t)} \dist _{x}\rangle\geq
0$, the last inequality is because of the symmetry of $F$ locally,
we have that $\Uparrow_{q}^{x}$ are in the same half sphere as
$\uparrow_{q}^{\pi(q)}=\uparrow_{q}^{\partial F_{v_1}}$ in $\Sigma_q
F$, for any $q\in F$.

The lemma thus follows.
\end{proof}

\subsection{Extending $\dist_F$-gradient flows }

\ \\

Since $\dist_F$ is concave in $D=A\backslash F$, for each $x\in D$
there is a unique $\dist_F$-gradient curve from $x$. We call a
gradient curve maximal if it is not a proper subset of another
gradient curve. Note that any maximal gradient curve has empty
intersection with $F$. We will extend maximal gradient curves to
include points in $F$ so that each point in $F-(S,0)$ is contained
in two extended maximal gradient curve. This property plus the
simply connectedness of $A$ allow us to choose one such curve for
each point in $F-(S,0)$, such that we can define a ``flow",
$\Psi_a^t:(S,a>0)\to A$, by $\Psi_a^t((s,a))=\gamma_a(t)$, where
$\gamma$ is the chosen extended maximal gradient curve at $(s,a)$,
passing any given extended maximal gradient curve at any given
$(s_0,a)$. Our goal is to show that $\Psi_a^t$ is 1-Lipschitz.

To carry out the above, the key is to establish the local separation
property for $F-(S,0)$ (see Lemma 2.11) and the local $1$-Lipschitz
property for $\Psi_a^t$.

Before moving on, we need the following two lemmas.
\begin{lemma}
For any $q\in (S,a)\subset F_u,\  a\neq 0,\, B_F(q,r)$ separates
$B(q,r)$, for $r$ small enough, where $B_F(q,r)$ is a closed
$r$-ball in $F$.
\end{lemma}
\begin{proof}
For $r$ small enough, we can assume that $B(q,r)\cap F=B(q,r)\cap
F_{v_1}$. 
By the local version of Perelman's stability theorem ([24, 4.7]) and
Proposition (2.1.1), we can choose $r$ sufficiently small, such that
$B(q,r)$ is homeomorphic to a $r$-ball on $T_qA$, which is
homeomorphic to $D^n$, and $ B_F(q,r)$ is homeomorphic to a $r$-ball
on $T_qF$, which is homeomorphic to $D^{n-1}$. By definition,
$B_F(q,r)\cap Fr(B(q,r))=(Fr(B_F(q,r))$ in $F$), which is
homeomorphic to $S^{n-2}$. By considering the double of $B(q,r)$, we
get that $B_F(q,r)$ separates $B(q,r)$.
\end{proof}
Let $q,r$ be as in Lemma 2.11, and let $\bar U_q\subset B(q,r)$ be a
convex closed neighborhood of $q$. Then $B_F(q,r)$ separates $\bar
U_q$ into two components $G_{q1},G_{q2}$.

Since $\dist_F$ is concave in $D$, for $\dist_F$-gradient curves
$\alpha(t),\ \beta(t)$, we have that $|\alpha(t)\beta(t)|$ is
$1$-Lipschitz if there exists a minimal geodesic joining $\alpha(t)$
and $\beta(t)$, for any $t$, in the domain $D$. The following
property will guarantee the local $1$-Lipschitz property for
$\Psi_a^t$.

\begin{lemma}
For $x,y$ in the interior of the same component say
$G_{q1}\backslash \partial \bar U_q$ (denoted by $ G_{b1}^\circ$),
we have that $[xy]\cap F=\emptyset$.
\end{lemma}
In the proof of Lemma 2.12, we need the following theorem about the
relationship between the boundary of a convex subset as an
Alexandrov space and the set boundary.
\begin{theorem}[{\rm{\cite{[AKP]}}}]
Let $C\subset X\in\Alex^m(\kappa)$ be a convex closed subset
($C\in\Alex^m(\kappa)$ with the induced metric). If $C$ has a
nonempty interior, then $\partial C=Fr(C)\cup (C\cap
\partial X)$.
\end{theorem}
\begin{proof}[Proof of Lemma 2.12]
Based on the local separation and the convexity of $F_v$, it is easy
to check that $\bar G_{qi},i=1,2$ are convex. By Theorem 2.13, we
have that $F\cap \bar G_{qi}\subset
\partial\bar G_{qi}$. Hence $[xy]\cap F=\emptyset$, 
because any minimal geodesic between interior points of a convex set
does not intersect the boundary of the convex set ([24, 5.2]).
\end{proof}
We are now in a position for the construction of a nonexpanding map.

For $x\in D$, let $q\in (S,a)\subset F_u, u\in N'$, be a point such
that $|xq|=|xF|$. First suppose $a\neq 0$. We will construct
extended maximal gradient curves from $(S,a)$.

For $b\in (S,a)$, let $b_j\in G_{b1}^\circ$, such that $b_j\to b$.
Let $\gamma_b^j:[0,t_0]\to A,\ t_0\geq 0$ denote the
$\dist_F-$gradient curve from $b_j$. Since $\gamma_b^j$ are
equi-continuous, after passing to a subsequence, we can suppose that
$\gamma_b^j$ converge to $\gamma_b:[0,t_0]\to A$ with
$\gamma_b(0)=b$. Note that $\gamma_b$ doesn't depend on the choice
of $b_j$. Indeed, for any $c\in U_b\cap (S,a)$, let $\gamma_c$ be a
similarly constructed curve, namely, let $c_j\in G_{b1}^\circ$, such
that $c_j\to c$. The $\dist_F-$gradient curve from $c_j$,
$\gamma_c^j:[0,t_0]\to A$, after passing to a subsequence, converge
to $\gamma_c:[0,t_0]\to A$ with $\gamma_c(0)=c$. Then
$|\gamma_b(t)\gamma_c(t)|=\lim
\limits_{j\to\infty}|\gamma_b^j(t)\gamma_c^j(t)|, 0\leq t\leq t_0$.
Since $\gamma_b^j(0),\gamma_c^j(0)\in G_{b1}^\circ\subset  U_b$, by
Lemma 2.12, we have that $[\gamma_b^j(0)\gamma_c^j(0)]\cap
F=\emptyset$. Set $\sigma(s):[0,l]\to[\gamma_b^j(0)\gamma_c^j(0)]$
with arc-length parametrization. Consider the curves
$\sigma_t(s)=\Phi_{\dist_F}^t([\gamma_b^j(0)\gamma_c^j(0)](s))$,
where $\Phi_{\dist_F}^t$ is the standard $\dist_F$-gradient flow
defined in [30, 2.2]. Let $P_m=\{0=s_0<s_1<\cdots<s_m=l,
s_i-s_{i-1}=\frac{l}{m}\}$ be a partition of $[0,l]$, with $m$ large
enough, such that $\frac{l}{m}\leq \frac{\varepsilon}{10}$,  where
$\varepsilon=|[\gamma_b^j(0)\gamma_c^j(0)],F|$. If
$t\leq\frac{\varepsilon}{4}$, then
$[\sigma_t(s_i)\sigma_{t}(s_{i-1})]\cap F=\emptyset$, thus
$|\sigma_t(s_i)\sigma_{t}(s_{i-1})|\leq|\sigma(s_i)\sigma(s_{i-1})|$.
Let $m\to\infty$, we get that $|\gamma_b^j(t)\gamma_c^j(t)|\leq
\Length(\sigma_t)\leq
\Length(\sigma_0)=|\gamma_b^j(0)\gamma_c^j(0)|$, where $\Length()$
denotes the length of the curve. If $t>\frac{\varepsilon}{4}$, we
can repeat the procedure for
$[\sigma_{\frac{\varepsilon}{4}}(s_i)\sigma_{\frac{\varepsilon}{4}}(s_{i-1})]$.
Finally, we get that for any $0\leq t\leq t_0$,
$|\gamma_b^j(t)\gamma_c^j(t)|\leq|\gamma_b^j(0)\gamma_c^j(0)|$. When
$j\to\infty$, we get that $|\gamma_b(t)\gamma_c(t)|\leq|bc|$. In
particularly, we have that if $b=c$, then $\gamma_b=\gamma_c$.

For simplicity, in the following context, we will say that
$\gamma_b$ and $\gamma_c$ are in the same component with respect to
$U_b$, and we will call $\gamma_b$ an extended maximal gradient
curve, if $\gamma_b- b$ is maximal.

By now we can see that for each point on $(S,a)$ there exist exactly
two extended maximal gradient curves. Choose a point $y\in (S,a)$,
denote the two extended maximal gradient curves by
$\gamma_{y1},\gamma_{y0}$. For any other point say $z$, we will
denote the two extended maximal gradient curves from $z$ by
$\gamma_{z1},\gamma_{z0}$, such that $\gamma_{zi}$ are a
continuation of $\gamma_{yi}$ along $[yz]$, i.e., there is a
partition of $[yz]$, $P=\{y_0=y,y_1,\cdots,y_k=z\}$, such that
$\gamma_{y_j,i},\gamma_{y_{j+1},i}$ are in the same component with
respect to $U_{y_j}$. It doesn't depend on the choice of the
partition, since for another partition $P_1$, we can consider $P\cup
P_1$ to get the independency.

Since $\pi_1((S,a))=0$, similarly as the final part of the proof of
Key Lemma 0.7, we can see that the denoting doesn't depend on the
choice of $[yz]$. Thus we can define two ``flows"
$\Psi_{ia}^t:(S,a)\to A$, by $b\to\gamma_{bi}(t),b\in (S,a)$, and we
fix one passing $x$, which exists, since $[xq]$ is contained in an
extended maximal gradient curve, denoted by $\Psi_{a}^t:(S,a)\to A$.

\begin{lemma}
The map $\Psi_{a}^t:(S,a)\to A$ is $1$-Lipschitz.
\end{lemma}
\begin{proof}
For $b,c\in (S,a)$,  we need to show that $|bc|\geq
|\gamma_b(t)\gamma_c(t)|$. Dividing $[bc]$ into small pieces so that
the above local 1-Lipschitz property holds, one gets that $|bc|\geq
\Length (\Psi_a^t([bc]))\geq|\Psi_a^t(b)\Psi_a^t(c)|$.
\end{proof}

\subsection{Completion of the proof of Theorem A}
\begin{lemma}
Let $S_x=\Psi_{a}^{|xq|}((S,a))$. Then $\pi|_{S_x}:S_x\to S$ is an
isometry.
\end{lemma}
\begin{proof}
Let $i:S\to (S,a)$ denote the natural isometry. Define a map,
$H:(S,a)\times [0,l] \to (S,a)$, where $l=|xq|$, by $H(x,t)=i\circ
\pi\circ\Psi_{a}^t(x)$.  Since $|\Psi_a^t(x),\Psi_a^{t'}(x')|\leq
|\Psi_a^t(x),\Psi_a^t(x')| +|\Psi_a^t(x'),\Psi_a^{t'}(x')|\leq
|xx'|+|t-t'|$, $H$ is continuous. Since $H(*,0)=id$, $H(x,l)$ is
onto. Recall that given two 1-Lipschitz onto maps between two
compact metric spaces, $g:X\to Y$ and $h:Y\to X$, then $g$ and $h$
are isometries. Since $\Psi_{a}^t$ and $\pi$ are 1-Lipschitz, the
desired result follows.
\end{proof}

If $a=0$, consider $\Sigma_qA$, from the proof of Lemma 2.8, we get
that there are $u,v\in \phi_{[pq]}(N')$ ($u$ may be equal to $v$),
such that $|u\ua_q^x|=|v\ua_q^x|=\frac\pi2$. Let $x_i\in [xq],
y_i\in ((S, a_i)\cap \pi^{-1}(q))\subset F_{\phi^{-1}_{[pq]}(u)}$,
where $a_i\neq 0$, and $x_i\to q$, $y_i\to q$, as $i\to \infty$.
\begin{claim}
There exist $\bar y_i\in [x_iy_i]\cap F$, and $\bar y_i\not\in S$,
such that $[x_i\bar y_i]^{\circ}\cap F=\emptyset$.
\end{claim}

Suppose that $\bar y_i\in (S, b_i)$ with $b_i \neq 0$. Let
$\Psi_{b_i}^t:(S,b_i)\to A$ denote the ``flow" such that,
$\Psi_{b_i}^{t}(\bar y_i), t\leq \varepsilon$ ($\varepsilon$ small
enough), is in the same component (defined after Lemma 2.11) as
$[y_i'\bar y_i]\subset [x_i\bar y_i]$ with $y_i',\bar y_i$ close
enough, and let
$\gamma_i(t)=\Psi_{b_i}^t(\bar y_i)$. 
Then %
by the construction of $\gamma_i(t)$, one obtains that
$|\gamma_i(t),[x_ix](t)|\leq |\bar y_ix_i|$, thus $\gamma_i(t)\to
[xq](t)$, for any $0\leq t\leq |xq|$, as $i\to\infty$. Hence by
passing to a subsequence we can suppose that
$\Psi_{b_i}^{|xq|}((S,b_i))\overset{isom}{\cong}S$ converge to
$S_x\overset{isom}{\cong}S$, with $x\in S_x$. 
Since $\pi|_{\Psi_{b_i}^{|xq|}((S,b_i))}$ is an isometry, so is
$\pi|_{S_x}$.

\begin{proof}[Proof of Claim 2.16]
It suffices to show that if $z\in [x_iy_i]\cap F$, then $z\not \in
S$. Argue by contradiction, suppose $z\in S$. Since $|x_iq|\leq
|x_iz|$ and $|y_iq|\leq|y_iz|$, we have that $[x_iq]\cup [qy_i]$ is
a minimal geodesic, which is impossible, because by the choice of
$y_i$ we have that $\ua_q^{x_i}\perp \ua_q^{y_i}$.
\end{proof}
Now we are ready for the proof of Theorem A.

\begin{proof}[Proof of Theorem A]
Case 1:  $\mathscr{F}=A$. See the proof following Lemma 1.1.

Case 2:  $\mathscr{F}\neq A$. By Lemma 1.5, we can assume that
$\pi_1(S)=0$. For $x\in A$ and for any $\bar y\in S$. Let
$y=
S_x\cap \pi^{-1}(\bar y)$. Since $\pi|_{S_x}:S_x\to S$ is an
isometry, $|xy|=|\bar x\bar y|$, and this shows that $\pi:A\to S$ is
a submetry.
\end{proof}
\section{Application}

We will prove Corollary 0.6, and we will show that in Theorem A if
$\mathscr{F}=A$, then $\pi$ is a bundle map. First we recall the
following lemma.
\begin{lemma}
Let $A\in \Alex^4(\kappa)$. Then the following statements are
equivalent:\\
(1) $A$ is topologically nice.\\
(2) $A$ is topologically regular.\\
(3) $A$ is a topological manifold.
\end{lemma}
Lemma 3.1 is pointed out by Kapovitch in \cite{[Ka2]}, for
completeness we include a proof here.

\begin{proof}
It is obvious that (1) $\Rightarrow$ (2) and (2) $\Rightarrow$ (3).

(3) $\Rightarrow$ (1): Any manifold point $p\in A$ satisfies that
$\pi_1(\Sigma_pA)=0$ (cf. [41, Theorem 1.1 (2)]), and that
$H_*(\Sigma_v\Sigma_pA)=H_*(S^2)$ ([41, Proposition 3.1]). Hence
$\Sigma_v\Sigma_pA$ is an $S^2$, and therefore $\Sigma_pA$ is a
simply connected manifold. By 3-dimension Poincar$\acute{\rm{e}}$
conjecture, $\Sigma_pA$ is homeomorphic to $S^3$, and thus $A$ is
topologically nice.
\end{proof}
\begin{proof}[Proof of Corollary 0.6]
Argue by contradiction, suppose that $\dim(S)> 0$. By [38, 9.8],
$\dim(S)\neq 1,3$. If $\dim(S)=2$, we may assume that $q\in A$ such
that the curvature are positive in $B(q,r)$. Then the modified
Busemann function $-\exp^{f}$ is strictly concave (\cite{[CDM2]}),
and thus $\pi|_{B(q,r)}:B(q,r)\to B(\pi(q),r)$ is strictly distance
decreasing; a contradiction to Theorem A.
\end{proof}
\begin{corollary}
Let the assumptions be as Theorem A. Suppose $\mathscr{F}=A$. Then
$\pi:A\to S$ is a bundle map with fiber $\mathbb R^2$.
\end{corollary}
Note that Corollary 3.2 doesn't hold, if one removes the condition
of topologically nice see [42, 14.8].

In the proof, we shall apply the following theorem which is a
sufficient condition for a bundle map:

\begin{theorem}[{\rm{\cite{[Se]}}}]
Let $X,\ X_0$ be two metric spaces and let $f:X\to X_0$ be a
continuous onto map. Suppose that for any $p\in X_0$ and
$\varepsilon>0$ there exits a $\delta(p,\varepsilon)>0$ such that
for any $q\in B(p,\delta)$, there is a homeomorphism $h:f^{-1}(p)\to
f^{-1}(q)$, with $|h(x)x|\leq \varepsilon$, for any $x\in
f^{-1}(p)$. If the fiber $F$ is locally compact and separable, and
the homeomorphism group of $F$ (with some natural topology) is
locally path connected, then $F$ is a Serre fibration. If in
addition, $X_0$ is finite dimensional $ANR$, then $f$ is a locally
trivial bundle map.
\end{theorem}
\begin{proof}[Proof of Corollary 3.2]
Since $\mathscr{F}=A$, for any $p, q\in S$ and a fixed $[pq]$, there
is a homeomorphism $h:\,\pi^{-1}(p)\to \pi^{-1}(q)$,  such that
$[xp],[pq],[qh(x)]$ determine a unique flat rectangle (Lemma (1.1.2)
and Proposition (2.1.3)). Then $|h(x)x|=|pq|$.

Since $A$ is topologically nice and $\pi$ is a submetry, if $p\in S$
is a regular point (i.e., $T_pS$ is isometric to $\mt R^{n-2}$),
then $\pi^{-1}(p)$ is a topological manifold ([32, Theorem D], note
that the proof is local, and thus apply to non-compact cases). Since
$\pi:\pi^{-1}(p)\to p$ is a deformation retraction,
$\pi^{-1}(p)\overset {homeo}{\simeq}\mt R^2$.

And by [11, 7.3], we know that the homeomorphism group
$homeo(\pi^{-1}(p))$ is locally path connected with the topology in
Theorem 3.3. By now we are able to apply Theorem 3.3 to conclude
that $\pi$ is a bundle map.
\end{proof}
\begin{remark}
Inspired by \cite{[Wo]}, it seems that for Case 2 of Theorem A, when
$\pi_1(A)=0$, $A$ isometrically splits.
\end{remark}

\section{Structure of Space of Directions}
Our main efforts in this section is to prove Proposition 2.1, and
thus complete the proof of Theorem A. We point out that Theorem 4.4,
which classifies certain isometric class in $\Alex^n(1)$, may have
independent interest.

First we recall the following:
\begin{definition}[{\rm{\cite{[BGP]}}}]
Let $X,\ Y\in \Alex(1)$. The join of $X$ and $Y$, $X*Y$, is defined
by $X\times Y \times [0,\frac{\pi}{2}]/\!\sim$, where $(x,y_1,0)\sim
(x,y_2,0)$ and $ (x_1,y,\frac{\pi}{2})\sim(x_2,y,\frac{\pi}{2})$,
with the metric: $\cos d((x_1,y_1,t),(x_2,y_2,s))=\cos t\cos
s\cos(|x_1x_2|)+\sin t\sin s\cos (|y_1y_2|)$.
\end{definition}
Then $X*Y\in \Alex(1)$ and $C(X*Y)=C(X)\times C(Y)$ (\cite{[BGP]}).
\begin{example}
$S^m*S^n=S^{m+n+1}$, where all the spheres are with the standard
metric with constant curvature $1$.
\end{example}

\begin{remark}
For the convenience of following use, we will make the following
conventions:

(1) Let 0-dim Alexandrov space with curvature bounded below by 1
without boundary be a space including two points with distance
$\pi$, one point is regraded as with boundary.

(2) If a subspace with restricted metric is isometric to an
Alexandrov space, we also say that it is convex, although when the
dimension is 0, there may not be a minimal geodesic in the subspace
joining two given points.

(3) When we say that two metric spaces are equal, we always mean
metrically, except otherwise stated. If there is no confusion, we
will not mention the metric.
\end{remark}

Let $A\in\Alex^n(1)$ and let $C\subset A$ be a closed convex subset
without boundary. Let $\hat \Sigma_pC=\{v\in \Sigma_pA\ |\
|v\Sigma_pC|\geq\frac{\pi}{2}\}$, and let $\hat T_pC=C(\hat
\Sigma_pC)$. Observe that if $A$ and $C$ are  Riemannian manifolds,
then $T_pA$ isometrically splits, i.e., $T_pA=T_pC\times \hat T_pC$
or equivalently $\Sigma_pA$ is isometric to the join of $\Sigma_pC$
and $\hat \Sigma_pC$ (\cite{[BGP]}). In Alexandrov geometry, such
property doesn't hold.

Observe that for $\Sigma\in \Alex^n(1)$ and $\Sigma_0,\
\Sigma_1\subset \Sigma$ convex closed subsets, if
$\Sigma=\Sigma_0*\Sigma_1$, then $\Sigma_0,\ \Sigma_1$ satisfy the
following conditions:

\noindent (1) For every $v\in\Sigma$, $v\in[v_0v_1]$, for some
$v_0\in \Sigma_0, \ v_1\in \Sigma_1$.

\noindent (2) $\dim(\Sigma_0)+\dim(\Sigma_1)+1= n$. ($\dim$ denotes
the Hausdorff dimension.)

In the following situation, conditions (1) and (2) are indeed
sufficient conditions for join.
\begin{theorem}
Let $\Sigma\in \Alex^n(1)$, let $\Sigma_0\subset \Sigma$ be a closed
convex subset without boundary, and let $\Sigma_1=\{v\in \Sigma\
|\hspace{2mm}|v\Sigma_0|\geq\frac{\pi}{2}\}$. Assume that
$\Sigma_1\neq \emptyset$ and $\Sigma_0,\ \Sigma_1$ satisfy (1) and
(2). If $\Sigma$ is topologically nice and homeomorphic to a sphere,
then $\Sigma=\Sigma_0*\Sigma_1$ 
and $\Sigma_i, i=0,1$, are topologically nice and homeomorphic to
spheres.
\end{theorem}
Theorem 4.4  was related to and inspired by the following Conjecture
4.5 made by Yamaguchi [42, 14.6].


\begin{conjecture}
Let $A\in \Alex^n(0)$ be noncompact with a soul $S$. If $p\in
(C_i-\partial C_i)$ is a topologically regular point, then $T_pX$ is
isometric to the product $T_pC_i\times K$, where $C_i$ are as in
1.1, and $K$ is a Euclidean cone.
\end{conjecture}

Yamaguchi proved Conjecture 4.5 for $n=4$ ([42, Theorem 14.5]);
where topologically nice and topologically regular are equivalent.

Using Theorem 4.4, with additional argument we can get the following
classification result, 
which can imply Proposition 2.1.

\begin{theorem}
Let $\Sigma\in \Alex^n(1),n\geq 2$. Let $\Sigma_0\subset \Sigma$ be
a convex closed subset without boundary with dimension $n-2$, and
let $\Sigma_1=\{u\in \Sigma\
|\hspace{2mm}|u\Sigma_0|\geq\frac{\pi}{2}\}$. Suppose that
$\Sigma_1\neq \emptyset$, $\Sigma$ is homeomorphic to a sphere and
topologically nice, and $\Sigma$ satisfies that for any $u\in
\Sigma,\ |u\Sigma_0|\leq \frac{\pi}{2}$ and
$|u\Sigma_1|\leq\frac{\pi}{2}$. Then $\Sigma_1$ is convex,
$\Sigma_0$ is homeomorphic to $S^{n-2}$ and topologically nice, and
$[\Sigma_0\Sigma_1]= \Sigma_0*\Sigma_1$, where $[\Sigma_0\Sigma_1]$
is with the restricted metric. Explicitly, we get the following
classification:

1) If $\Sigma_1=S^1(r)$ with $r\leq 1$, then
$\Sigma=\Sigma_0*\Sigma_1$.

2) If $\Sigma_1=\{v_1,v_2\}$ with $ |v_1v_2|=\pi$, then
$\Sigma=\Sigma_1*\Sigma_1^{\perp}$.

3) If $\Sigma_1=\{v\}$, then $\Sigma=D(\Sigma_1*\hat\Sigma_1)$; the
double of half suspension of $\hat \Sigma_1=\{u\in
\Sigma\,\mid\,|u\Sigma_1|\geq\frac{\pi}{2}\}$.

4) If $\Sigma_1=[ab]$, then $\Sigma=(a*\hat
\Sigma_1\cup_{\hat\Sigma_1}
b*\hat\Sigma_1)\cup_\partial(\Sigma_1*\Sigma_0)$ with the gluing
metric ([2, Definition 3.1.12]).
\end{theorem}

In the rest of this section, we will prove Theorem 4.4 and Theorem
4.6. First we will present some preparations.
\begin{lemma}
\label{lem:dim} Let $\Sigma\in \Alex^n(1)$ and let
$\Sigma_1,\Sigma_0\subset\Sigma$ be two convex closed subsets with
dimension $k,l$ respectively. Suppose that $\Sigma$ satisfies: for
any $v_0,v_1\in\Sigma_0,\Sigma_1$ respectively,
$|v_0v_1|=\frac{\pi}{2}$. Then $n\geq l+k+1$.
\end{lemma}
\begin{proof}
Let $C_i\subset \Sigma_i\backslash \partial\Sigma_i,i=1,0$, be two
closed convex subsets. Note that $[C_1C_0]$ with the restricted
metric is a closed subset of $\Sigma$. It suffices to construct a
nonexpanding map from $[C_1C_0]$ to $ C_1*C_0$. If so,
$\dim(\Sigma)\geq \dim([C_1C_0])\geq \dim(C_1*C_0)= k+l+1$, where
the last identity is because, by definition, we have that
$\dim(C(C_0*C_1))=\dim(C(C_0)\times C(C_1))= k+l+2$.

First we claim that  $[x_1y_1]\cap [x_2y_2]\subset\{x_1,y_1,x_2,y_2\}$, 
for $x_1,x_2\in C_1,y_1,y_2\in C_0$. Thus for $p\in [C_1C_0]$, $p$
can be uniquely written as $([xy],s)$, where $x\in C_1,\;y\in
C_0,\;p\in[xy]$ and $s=|px|$. Then we can construct a map,
$g:[C_1C_0]\to C_1*C_0$, by $g|_{C_i}=id,i=0,1$ and
$g(([pq],s))=([pq],s)$ (the unique point on the unique minimal
geodesic $[pq]$ with distance $s$ from $p$), for $p\in C_1$, $q\in
C_0$.

Next we will show that $g$ is nonexpanding. In order to do so we
need the following fact:

For $p\in C_1$ and $q\in C_0$, $|pq|=|p\Sigma_0|$. By Lemma 1.2,
$|\ua_q^pv|=\frac\pi2$, for any $v\in\Sigma_q\Sigma_0$. Hence
$\tilde\angle(p,q,r)=\angle(p,q,r)=\frac\pi2$, for any
$r\in\Sigma_0$. By Lemma 2.6, there is a convex triangle isometric
to the corresponding one on space form.

For $x,y\in [C_1C_0]$, by the property of $[C_1C_0]$, we have that
$x\in [p_1p_0],y\in[q_1q_0]$. Let $|xp_1|=s$ and $|yq_1|=t$. Then by
the above fact, $\cos(|p_1y|)=\cos(|p_1q_1|)\cos t,\
\cos(|p_0y|)=\cos(|p_0q_0|)\sin t$. By the monotonicity of angle,
$|xy|\geq \cos s\cos t\cos (|p_1q_1|)+\sin s\sin t\cos
(|p_0q_0|)=|g(x)g(y)|$, which shows that $g$ is  nonexpanding.

Finally, we will verify the claim. It suffices to show that if
$x_1\neq x_2,y_1\neq y_2$, then $[x_1y_1]\cap [x_2y_2]=\emptyset$.
Suppose that $[x_1y_1]\cap [x_2y_2]=z$. Hence by the above fact,
$[zy_1]\subset $ the totally geodesic triangle bounded by
$\{x_2,y_1,y_2\}$, which is isometric to the corresponding triangle
on $S^2(1)$, or the geodesic will branch. Thus $\angle
(x_1,y_1,y_2)<\frac\pi2$. This is a contradiction.
\end{proof}

\begin{remark}
Let $\Sigma\in \Alex^n(1)$, and let $\Sigma_0\subset \Sigma$ be a
convex closed subset. Suppose $\Sigma_1=\{v\in \Sigma\
|\hspace{2mm}|v\Sigma_0|\geq\frac{\pi}{2}\}\neq \emptyset$. If
$\partial \Sigma_0=\emptyset$ or $\partial\Sigma_1=\emptyset$, then
by Lemma 1.2, $\Sigma$ satisfies the conditions of Lemma 4.7. These
are the cases usually used in the following context.
\end{remark}
\subsection{Proof of Theorem 4.4}
\ \\

Sketch of the proof: we will prove Theorem 4.4 by induction on the
dimension of $\Sigma$. And we will prove the inductive step
according to the different situations of the dimension of $\Sigma_0$
and $\Sigma_1$. Except two simple cases, we will first show that for
each point in $\Sigma_1$ and each point in $\Sigma_0$, there exist
exactly $m$ minimal geodesics joining them for some $m>0$. Then
argue by contradiction, we can get that $m=1$, using this we can
derive that $\Sigma$ is a join.

\begin{proof}[Proof of Theorem 4.4]
Denote $\dim(\Sigma_1)$ by $I$ and $\dim(\Sigma_0)$ by $J$. First we
will prove the following property.
\begin{sub lemma}
We have that $\Sigma_1$ is convex.
\end{sub lemma}
\begin{proof}
For $p,q\in \Sigma_1$, when $|pq|<\pi$, by triangle comparison we
have that for any $ x\in\Sigma_0,\ |x,[pq]|\geq\frac{\pi}{2}$. Hence
$[pq]\subset \Sigma_1$. When $|pq|=\pi$, if there is no other points
in $\Sigma_1$, then the sublemma holds. If there is
$x\in\Sigma_1-\{p,q\}$, one can easily
see that $|px|+|xq|=\pi=|pq|$, 
thus $[px],\ [xq]\subset \Sigma_1$, and  $[px]\cup[xq]$ is a minimal
geodesic which lies in $\Sigma_1$. The sublemma thus follows.
\end{proof}
We then proceed the proof by induction on $n$. For $n=1$, i.e.,
$I=0$ and $J=0$, clearly the theorem holds. Suppose $n-1$ the
theorem holds.

Now we will prove the inductive step according to different situations of $I$ and $J$:\\

If $J=0$, i.e., $\Sigma_0=\{v_1,v_2\}$ with $|v_1v_2|=\pi$, then
$\Sigma=\Sigma_0*\Sigma_1$. Thus $\Sigma_{v_1}=\Sigma_1$. Because
$\Sigma$ is topologically nice, $\Sigma_1$ is homeomorphic to
$S^{n-1}$ and topologically nice.\\

Hence in the following we can assume that $J>0$.

Next we want to show that $\partial\Sigma_1=\emptyset$. In order to
apply inductive assumptions to $\Sigma_p\Sigma$, for $p\in\Sigma_0$,
it suffices to check that
$\Sigma_p\Sigma_0,(\Sigma_p\Sigma_0)^{\perp}$ satisfy conditions (1)
and (2). For condition (1):
\begin{sub lemma}
For any $p\in \Sigma_0$ and for any $\bar v\in\Sigma_p\Sigma,\bar
v\in[\bar v_1\bar v_0]\subset[\Ua_p^{\Sigma_1}\Sigma_p\Sigma_0]$.
\end{sub lemma}

Sublemma 4.10 implies that for $p\in
\Sigma_0,(\Sigma_p\Sigma_0)^{\perp}=\Uparrow _p^{\Sigma_1}$ with the
restricted metric in $\Sigma_p$.

\begin{proof}[Proof of Sublemma 4.10]
Observe that by Lemma 1.2, we have that for every $p\in\Sigma_1$ and
every $q\in\Sigma_0$, $|pq|=\frac\pi2$. Hence for every
$r\in\Sigma_0,\angle(p,q,r)=\tilde\angle(p,q,r)$.
By Lemma 2.6, we have that there exists a totally geodesic triangle
determined by $[pq]$ and $[qr]$, which is isometric to a geodesic
triangle in $S^2(1)$.

Let $\bar v\in(\Sigma_p\Sigma)'$. Then $\bar v=\ua_p^x$ for some
$x\in\Sigma$. By perturbing $x$ along $[px]$, we can suppose that
$[px]$ is unique. From condition (1), $x\in[x_1x_0]$, for some
$x_i\in\Sigma_i$. Therefore there exists a totally geodesic triangle
$T$ determined by $[x_0p],[x_1p]$, which is isometric to a geodesic
triangle in $S^2(1)$., and by the uniqueness, we get that
$[px]\subset T$. Thus $\bar v\in[\ua_p^{x_1}\ua_p^{x_0}]$. Since
$(\Sigma_p\Sigma)'$ is dense in $\Sigma_p\Sigma$, we get the
Sublemma.
\end{proof}

For condition (2): from the construction, we see that, $\psi_1:\
\Uparrow_p^{\Sigma_1}\to\Sigma_1,\psi_1(\Uparrow_p^x)=x,x\in\Sigma_1$,
is a submetry (since $\psi_1$ has the horizontal lifting property,
as can be seen in the proof of Sublemma 4.10).
Thus the map $\psi_1$ is $1$-Lipschitz. Therefore
$\dim(\Uparrow_p^{\Sigma_1})\geq \dim(\Sigma_1)$. Hence on one hand,
we have
\begin{center}
$\dim(\Sigma_p\Sigma)=\dim(\Sigma_0)-1+\dim (\Sigma_1)+1 \leq
\dim(\Sigma_p\Sigma_0)+\dim(\Uparrow_p^{\Sigma_1})+1.$
\end{center}
On the other hand, by Lemma 4.7, we have the opposite inequality.
Thus
\begin{center}
$\dim(\Sigma_p\Sigma)=\dim(\Sigma_p\Sigma_0)+\dim(\Uparrow_p^{\Sigma_1})+1$.\\
\end{center}

By now, we obtain that $\Sigma_p\Sigma,p\in\Sigma_0$ satisfies the
inductive assumptions. Hence for every $p\in\Sigma_0,
\Sigma_p\Sigma=\Sigma_p\Sigma_0*(\Sigma_p\Sigma_0)^{\perp}$, and
$\Sigma_p\Sigma_0$ is homeomorphic to a sphere. Thus $\Sigma_0$ is a
topological manifold.


For the following use, we recall Alexander duality \cite{[Ha]}:
Let $K\subset S^n$ be a compact, locally contractible, nonempty,
proper subspace. Then $\widetilde H_i(S^n-K;\mathbb{Z})=\widetilde
H^{n-i-1}(K;\mathbb{Z})$, for all $i$, where $\widetilde
H_i,\widetilde H^i$ denote the reduced homology and cohomology.

Now we are ready to show that $\partial\Sigma_1=\emptyset$.
Argue by contradiction, assume $\partial \Sigma_1\neq\emptyset$.
Then $\Sigma_1$ is contractible, since $\Sigma_1$ has positive
curvature (cf. \cite{[Per1]}). From condition (1), we see that
$\dist_{\Sigma_1}$ has no critical points in
$\Sigma-\Sigma_1\cup\Sigma_0$. Thus we have that there is a
deformation retraction from $\Sigma-\Sigma_1$ to $\Sigma_0$.
Therefore $\widetilde H^{n-1-i}(\Sigma_1)=0$. On one hand by
Alexander duality $\widetilde H_i(S^n-\Sigma_1)=\widetilde
H_i(\Sigma_0)=\widetilde H^{n-1-i}(\Sigma_1)=0$, for any $i\geq 0$.
On the other hand, since $\Sigma_0$ is a topological manifold, it
follows that if $\Sigma_0$ is orientable, $\widetilde
H_n(\Sigma_0)\neq 0$, and if $\Sigma_0$ is not orientable,
$\widetilde H_{n-1}(\Sigma_0)\neq 0$. This is a contradiction.
Therefore $\partial\Sigma_1=\emptyset$.


Now we come to the second case:\\

If $I=0$, we have that $\Sigma_1=\{w_1,w_2\}$ with $|w_1w_2|=\pi$.
Then $\Sigma=\Sigma_0*\Sigma_1$. Because $\Sigma$ is topologically
nice, $\Sigma_0$ is homeomorphic to $S^{n-1}$ and
topologically nice.\\

Hence in the following we can suppose that $I>0$. First we will show
that for any $x_i\in\Sigma_i,i=0,1$, there are $m$ minimal geodesics
joining $x_0$ and $x_1$, for some $m>0$.

Similarly, we have that for every $ x\in \Sigma_1$,
\begin{center}
$\dim(\Sigma_x\Sigma)=\dim(\Sigma_x\Sigma_1)+\dim(\Uparrow_x^{\Sigma_0})+1$.
\end{center}
Hence $\dim(\Uparrow_x^{\Sigma_0})=\dim(\Sigma_0)$.

Since $\partial \Sigma_1=\emptyset $, we can apply Sublemma 4.10 to
$\Sigma_x\Sigma$, $x\in \Sigma _1$. We get that
$\Sigma_x\Sigma_1,(\Sigma_x\Sigma_1)^\perp$ satisfy condition (1)
and $\Uparrow_x^{\Sigma_0}=(\Sigma_x\Sigma_1)^\perp$. Hence
$\Uparrow_x^{\Sigma_0}$ is convex, and similarly,
$\psi_0:\Uparrow_x^{\Sigma_0}\to \Sigma_0,\psi_0(\Uparrow_x^p)=p,
p\in\Sigma_0$, is a submetry. Thus for a $(J,\delta)$-burst point
(see [3, 5.2]) $y\in \Sigma_0$, there is a neighborhood $U_y$ of
$y$, such that $\psi_0^{-1}(U_y)\overset{homeo}{\simeq} U_y\times
F_0$, where $F_0$ is a $0-\dim\ MCS$-space (\cite{[Per1]}). Hence
$F_0$ is a collection of discrete points. Since
$\Uparrow_x^{\Sigma_0}$ is compact, $\#\{\psi_0^{-1}(y)\}<\infty$,
say $m$. In the following, we will show that $m$ is independent of
the choice of $y$.
\begin{sub lemma}
For every $ x\in \Sigma_1$ and every $p,q\in \Sigma_0,\
\#\{\Ua_x^p\}=\#\{\Ua_x^q\}$.
\end{sub lemma}
\begin{proof}
Since for a fixed minimal geodesic $[pq]$ and for any $[xp]$, there
is a totally geodesic triangle with two sides $[pq],[xp]$, which is
isometric to a geodesic triangle in $S^2(1)$, there is a
corresponding $[xq]$. This determines a map from $\Ua_x^p$ to
$\Ua_x^q$, which is 1-1, onto, or there will be a contradiction to
the join of $\Sigma_p$ or $\Sigma_q$, and thus the sublemma follows.
\end{proof}

Then we want to prove the following:
\begin{sub lemma}
The map $\psi_0$ is an $m$-fold locally isometric covering map.
\end{sub lemma}
\begin{proof}
Let $\psi_0^ {-1 }(p)=\cup _{i=1,\cdots,m}\{p_i\}$ and
$\varepsilon=\min\limits_{1\leq i,j\leq m}\{|p_ip_j|\}$. Then
$B(p_i,\frac{\varepsilon}{5})\cap
B(p_j,\frac{\varepsilon}{5})=\emptyset$. Since if there is $x\in
B(p_i,\frac{\varepsilon}{5})\cap B(p_j,\frac{\varepsilon}{5})$, then
$|p_ip_j|\leq |xp_i|+|xp_j|\leq \frac{\varepsilon}{2}$, a
contradiction. Hence $\psi_0^{-1}(B(p,\frac{\varepsilon}{5}))=
B(\psi_0^{-1}(p),\frac{\varepsilon}{5})=B(\cup
\{p_i\},\frac{\varepsilon}{5})=\cup B(p_i,\frac{\varepsilon}{5})$.

Next we will show that $B(p,\frac{\varepsilon}{5})$ is isometric to
$ B(p_i,\frac{\varepsilon}{5})$, for any $1\leq i\leq m$. For any
$\bar x\in B(p,\frac{\varepsilon}{5})$, there exists a unique $x\in
B(p_i,\frac{\varepsilon}{5})$ such that $|p_ix|=| p\bar x|$, or will
contradict to Sublemma 4.11. For $\bar y\in
B(p,\frac{\varepsilon}{5})$, let $y\in B(p_i,\frac{\varepsilon}{5})$
such that $|yx|=|\bar x\bar y|$. Then $|p_iy|=|p\bar y|$, or
$|p_ip_j|\leq \frac{4\varepsilon}{5}$, a contradiction. Thus we get
the sublemma.
\end{proof}

Since we have showed that for any $x\in\Sigma_1$, $\Sigma_x\Sigma$
satisfies the conditions of the theorem, 
by induction assumption, $\Sigma_x\Sigma=\Sigma_x\Sigma_1*
\Uparrow_x^{\Sigma_0}$, and $\Sigma_x\Sigma_1,
\Uparrow_x^{\Sigma_0}$ are homeomorphic to spheres.

Hence $\Sigma_0\overset {homeo}{\simeq}S^J/\Gamma_0$.
Similarly, we get that for any $p\in \Sigma
_0,\psi_1:\Uparrow_p^{\Sigma_1}\to \Sigma_1 $ is an $m$-fold
covering map, and therefore $\Sigma_1\overset
{homeo}{\simeq}S^I/\Gamma_1$. By now we can see that for any
$x_i\in\Sigma_i$, there are $m$
minimal geodesics joining $x_0$ and $x_1$.\\

Next, we want to show that $m=1$. 

\begin{sub lemma}
Let $X_i=\{v\hspace{2 mm}|\hspace{2 mm}|v\Sigma_i|\leq
\frac{\pi}{4}\}$, $i=0,1$. Then $X_i$ are locally trivial bundles
over $\Sigma_i$.
\end{sub lemma}
\begin{proof}
Without loss of generality, we can suppose that $i=0$. First observe
that there is a natural map, $g:X_0\to\Sigma_0$, defined in the
following way: for any $y\in X_0$, by condition (1), we have that
$y\in[y_1y_0]$, for some $y_i\in\Sigma_i$, set $g(y)=y_0$.

For $p\in \Sigma_0$, let $U$ be a locally isometric neighborhood of
$p$ as in Sublemma 4.12. Then $U$ satisfies that the covering map
$\psi_0:\Uparrow_x ^{\Sigma_0} \to \Sigma_0$, for any
$x\in\Sigma_1$, when restricted to each component of $\psi_0
^{-1}(U)$, is an isometry.
Indeed, if $\Sigma_0$ is an $S^1$, then nothing need to say. If
$\Sigma_0$ is not an $S^1$, let $Y_1,Y_2$ be two universal covering
spaces of $\Sigma_0$, with covering maps $c_1,c_2$. There exists an
isometry, $g:Y_1\to Y_2$, which is fiber preserving, i.e., $c_2\circ
g=c_1$. Then we can get a neighborhood $V$, such that if $c_1$, when
restricted to $c_1^{-1}(V)$, is an isometry, so is $c_2$.

Now we can define a map, $\eta: g^{-1}(U)\to U\times F$, where
$F=g^{-1}(p)$, in the following way: for $z\in g^{-1}(U)$, then
$z\in[z_1z_0]$ for some $z_i\in\Sigma_i$, set $\eta
(z)=(z_0,([z_1p],|zz_1|))$, where $([z_1p],|zz_1|)$ represents a
point on the geodesic $[z_1p]$ with distance $|zz_1|$ from $z_1$,
and $[z_1p]$ is the geodesic such that
$[\uparrow_{z_1}^p\uparrow_{z_1}^{z_0}]$ is the horizontal lifting
of $[pz_0]$, i.e., $|\uparrow_{z_1}^p\uparrow_{z_1}^{z_0}|=|pz_0|$,
which is unique, by the choice of $U$. Then $\eta$ is injective and
onto. Next we will check that $\eta$ is continuous. For points
$z_i\in[z_i^1z_i^0],z_i\to y\in[y_1y_0]$, since $|py_0|\leq
|\uparrow_{y_1}^p\uparrow_{y_1}^{y_0}|\leq\lim\inf
|\uparrow_{z_i^1}^p\uparrow_{z_i^1}^{z_i^0}|=\lim\inf
|pz_i^0|=|py_0|$, we can get the continuity of $\eta$. On the other
hand, we can see that the inverse of $\eta$, $ \bar\eta :U\times
F\to g^{-1}(U)$, is defined by $\bar\eta((x_0,f))=([f_1x_0],|ff_1|),
\ \text{where}\ f\in[pf_1], \ \text{and}\
[\uparrow_{f_1}^{x_0}\uparrow_{f_1}^p]$ is the horizontal lifting of
$[x_0p]$. Likewise, we can see that $\bar\eta$ is injective, onto
and continuous. Hence $\eta$ is a homeomorphism, and obviously
$p_2\circ \eta=g$, where $p_2:U\times F\to F$ is the projection to
the second factor. Thus we get the Sublemma.
\end{proof}

Thus $X_1$ is $D^{J+1} - bundle\ over\ S^I/\Gamma_1$ and $X_0$ is
$D^{I+1}- bundle\ over\ S^J/\Gamma_0$.\\

If $I=1$ and $J=1$ (cf. [42, Theorem 14.5]), $\Sigma$ is glued by
two solid tori (since by \cite{[Ha]} we know that Klein bottle can
not be embedded into $S^3$). Hence $\Sigma$ is a Lens space with
$\pi_1(\Sigma)=\mathbb
Z_m$. Therefore $m=1$.\\

If $I=1$ and $J>1$, by the long exact sequences:
\[
\rightarrow \pi_1(S^J)\rightarrow \pi_1(\partial X_1)\rightarrow
\pi_1(S^1)\rightarrow 0,
\]
\[\rightarrow \pi_1(S^1)\rightarrow
\pi_1(\partial X_0)\rightarrow \pi_1(S^J/\Gamma_0)\rightarrow 0,
\]
we have that $\pi_1(\partial X_1)=\mathbb Z$ and $\Gamma_0$ is a
quotient group of $\mathbb Z$. Hence $\Gamma_0=\mathbb Z_m$. By
Alexander duality, we have that $\mathbb Z_m=\widetilde
H_1(S^J/\Gamma_0;\mathbb Z)=\widetilde H_1(S^n\backslash S^1;\mathbb
Z)=\widetilde H^{n-1-1}(S^1;\mathbb Z)=0$, where the last identity
is
because $n> 3$. Thus $m=1$. Similarly for $I>1$ and $J=1$.\\

If $I>1$ and $J>1$, denote $\widetilde X_i=\psi_i^*(X_i)$; the pull
back bundle by the covering map $\psi_i$. Then $\widetilde X_i$ is
an $m$-fold cover of $X_i$. The covering map is denoted by
$\phi_i,i=0,1$. Since $\pi_1(\widetilde X_i)=0$, $\widetilde X_i$ is
the universal cover of $X_i$.

Because $\partial\widetilde X_1$ and $\partial \widetilde X_0$ are
both universal coverings of $\partial X_1=\partial X_0$, there
exists a homeomorphism $h_0:\partial\widetilde X_1\to\partial
\widetilde X_0, $ such that $\phi_1=\phi_0\circ h_0$. Set $\tilde
X=\tilde X_1\cup \widetilde X_0/\sim$ (topologically), where $x\sim
h_0(x)$, when $x\in\partial \widetilde X_1$. We will show that
$\widetilde X$ is an $m-$fold cover of $X$.

Let $h:\widetilde X\to X$ be defined by:
$$h(x)=\left\{ \begin{array}{cl} \phi_1(x), & x\in \widetilde X_1,\\ \phi_0(x), &
    x\in\widetilde X_0.\end{array}\right.$$
If $x\in\partial \widetilde X_1$, $h(x)=\phi _1(x)=\phi_0(h_0(x))$.
Thus $h$ is well defined and continuous. Next we will check that $h$
is a covering map.

Since for $x\in X_1\cap X_0$, there exists a small neighborhood $U$
of $x$, denote $U_i=X_i\cap U$, such that for the closure $\overline
U_i$, we have that $\phi_i^{-1}(\overline U_i)=\cup_{j=1,\cdots,
m}\overline U_{ij},i=0,1$, with $\phi_i|_{\overline U_{ij}}$ are
homeomorphisms. Thus $h^{-1}(\overline U)=\cup_{j=1,\cdots,
m}(\overline U_{1j}\cup_{h_0}\overline U_{0j'})$. And it is clear
that $h|_{\overline U_{1j}\cup_{h_0}\overline U_{0j'}}$ is bijective
and onto, therefore a hoemomorphism (since $\overline
U_{1j}\cup_{h_0}\overline U_{0j'}$ is compact). Thus
$h|_{U_{1j}\cup_{h_0}U_{0j'}}$ is a hoemomorphism. It follows that
$\widetilde X$ is an $m-$fold cover of $\Sigma$, and
therefore $m=1$.\\

Finally, we will show that $\Sigma=\Sigma_0*\Sigma_1$. As can be
seen in the proof of Lemma 4.7, we have that for $ p\in \Sigma$, $p$
can be uniquely written as $([xy],s)$, where $x\in \Sigma_1,\;y\in
\Sigma_0$ and $s=|px|$. By now we can construct a map,
$\xi:\Sigma\to \Sigma_0*\Sigma_1$, such that $\xi|_{\Sigma_i},i=0,1$
are isometries, and maps $u=([xy],s)$ to $([\xi(x)\xi(y)],s)$, where
$s=|\xi(u)x|$. For every $p_1=([x_1y_1],s),\
p_2=([x_2y_2],t)\in\Sigma$, since $|x_1p_2|=|\xi(x_1)\xi(p_2)|$ and
$|y_1p_2|=|\xi(y_1)\xi(p_2)|$, by triangle comparison, we get that
$|p_1p_2|\geq|\xi(p_1)\xi(p_2)|$. Hence
$\angle(\uparrow_p^{z_1},\uparrow_p^{z_2})\geq\angle(\uparrow_{\xi(p)}^{\xi(z_1)},
\uparrow_{\xi(p)}^{\xi(z_2)})$, for $p\in \Sigma_1$ and
$z_i\in\Sigma$, that is, the induced map, $\xi_*:\Sigma_p\Sigma\to
\Sigma_{\xi(p)}(\Sigma_0*\Sigma_1)$, is 1-Lipschitz. Since
$\Sigma_p\Sigma$ is isometric to
$\Sigma_{\xi(p)}(\Sigma_0*\Sigma_1)$, we can see that $\xi_*$ is an
isometry. Specially,
$\angle(\uparrow_p^{p_1},\uparrow_p^{p_2})=\angle(\uparrow_{\xi(p)}^{\xi(p_1)},
\uparrow_{\xi(p)}^{\xi(p_2)})$. Then by hinge comparison, we have
that $|p_1p_2|\leq|\xi(p_1)\xi(p_2)|$. Thus we get that $\xi$ is an
isometry.
\end{proof}

\subsection{Proof of Theorem 4.6}
\ \\

Before we focus on the proof of Theorem 4.6, we collect some lemmas
which will be used in the proof of Theorem 4.6. We begin with the
following very simple lemma:
\begin{lemma}
Let $ X\in \Alex ^n(\kappa),\ n\geq 2$ and let $C\subset X$ be a
closed locally convex subset with $\dim(C)=n$. If $\partial C\neq
\emptyset$, suppose that $C\cap \partial X=\partial C$, then $C=X$.
\end{lemma}
\begin{proof}
First suppose that $\partial C=\emptyset$. We will proceed by
induction, if $n=1$, we can see that the lemma holds. For the
inductive step, argue by contradiction. Suppose that there is $x\in
X\backslash C$, let $y\in C$ be a point such that $|xy|=|xC|$. By
induction, we have that $\uparrow _y^x\in \Sigma_yC$. This is a
contradiction.

If $\partial C\neq\emptyset$, by considering the double of $X$, we
can get the desired result.
\end{proof}
\begin{lemma}
Let the assumptions be as in Theorem 4.6. 
If $\Sigma_1=S^1(r)$ with $r\leq 1$, then 
for any $v\in \Sigma$, there are $v_i\in \Sigma_i, i=0,1$, such that
$v\in[v_0v_1]$.
\end{lemma}
\begin{proof}
For $v\in \Sigma$, let $w\in\Sigma_0$ such that $|vw|=|v\Sigma_0|$.
Then by Lemma 1.2, we have that
$\uparrow_w^v\in(\Sigma_w\Sigma_0)^{\perp}
$. Since $\Uparrow_w^{S^1}=(\Sigma_w\Sigma_0)^{\perp}=S^1$, the
result follows.
\end{proof}

\begin{lemma}
Let $A\in \Alex^n(1),\ \partial A\neq \emptyset$. Then $A$ can't
contain a convex closed subset without boundary with positive
dimension whose intersection with $\partial A$ is empty.
\end{lemma}
\begin{proof}
Argue by contradiction. Suppose that there exists a convex closed
subset $C$ without boundary with positive dimension such that $C\cap
\partial A$ is empty. Let $h=\dist_{\partial A}$ and let $p\in C$ be
a point such that $h(p)=\min_{x\in C}\{h(x)\}$. By the first
variation formulae, $|\Uparrow_p^{\partial
A}\Sigma_pC|\geq\frac{\pi}{2}$. By Lemma 1.2, $|\uparrow_p^{\partial
A}v|=\frac{\pi}{2}$, for any $ v\in \Sigma_pC$. Then for any $x\in
C$, we have that $
0=d_ph(\uparrow_p^x)\geq\dfrac{h(x)-h(p)}{|px|}\Rightarrow h(p)\geq
h(x)$. Hence $h(x)=constant$, for any $x\in C$. Let $\gamma\subset
C$ be a minimal geodesic. By Theorem 0.5, there exists a flat
rectangle, a contradiction to $A\in \Alex^n(1)$.
\end{proof}
\begin{lemma}
Let $\Sigma\in \Alex^n(1)$. Suppose that $\Sigma$ is homeomorphic to
a sphere and  topologically nice. Let $\Sigma_0\subset \Sigma$ be a
convex closed subset with dimension $n-1$, without boundary. Then
$\Sigma$ is homeomorphic to $S(\Sigma_0)$, and $\Sigma_0$ is
homeomorphic to a sphere and topologically nice.
\end{lemma}

In the proof we will use the following lemma.

\begin{lemma}[{\rm{[24, 6.2]}}]
Let $C\in \Alex ^n(1)$ with $\partial C\neq \emptyset$. Let $x\in C$
be the unique point with maximal distance from $\partial C$. Then
$(C,\partial C)\overset{homeo}{\simeq }(\bar
C(\Sigma_xC),\Sigma_xC)$, where $\bar C(\Sigma_xC)\subset T_x$
denotes the closed  unit ball at the origin.
\end{lemma}

\begin{proof}[Proof of Lemma 4.17]
We will prove the lemma by induction on $n$. Clearly when $n=1$ the
lemma holds. Suppose $n-1$ the lemma holds.

First apply the inductive assumptions to $\Sigma_p\Sigma$ and
$\Sigma_p\Sigma_0$, for $p\in \Sigma_0$, we can see that $\Sigma_0$
is a topological manifold. Then we have that $\Sigma-\Sigma_0$ has
two components $H_1,H_2$, each is with set boundary $\Sigma_0$.
Observe that for every $x,y\in \bar H_i,\ [xy]\subset \bar H_i$, if
not, by the convexity of $\Sigma_0$, we can get a contradiction.
I.e., $\bar H_i$ are convex. 

By Theorem 2.13, we have that $\partial\bar  H_i=\Sigma_0,i=1,2$,
and by Lemma 4.18, we have that $(\bar H_i,\partial\bar
H_i)\overset{homeo}{\simeq }(\bar
C(\Sigma_v\Sigma),\Sigma_v\Sigma)$, where $v\in\bar H_i$ is the
point such that $|v\partial\bar H_i|=\max\{\dist_{\partial \bar
H_i}\}$. Hence
$\Sigma_0\overset{homeo}{\simeq}\Sigma_v\Sigma\overset{homeo}{\simeq
} S^{n-1}$, where the last one is by the topologically nice property
of $\Sigma$. The lemma thus follows.
\end{proof}
\begin{lemma}
Let $C\in \Alex^n(1)$ with  $\partial C=\emptyset$. Let $C_0\subset
C$ be a convex closed subset with dimension $n-1$ and without
boundary. Suppose that $C_0$ separates $C$. If there is $v\in C$
such that $|vC_0|\geq \frac{\pi}{2}$, then $[vC_0]\overset{\rm
isom}{\cong} S^+(C_0)$, where $[vC_0]$ is with the restricted
metric. In particular, for $x\in C$ with $|xv|<\frac{\pi}{2}$, we
have that $x\in[vv_0]$, for some $v_0\in C_0$.
\end{lemma}
\begin{proof}
Suppose that $C-C_0$ has two components $H_1,H_2$ and $v\in H_1$.
$\bar H_1$ is convex, as can be seen in the proof of Lemma 4.17. By
Theorem 2.13, $\partial \bar H_1=C_0$, and by Lemma 1.2,
$|vx|=\frac{\pi}{2}$, for any $x\in C_0$. Then $D(\bar
H_1)=S(C_0')$. By the structure of $D(\bar H_1)$, $C_0$ separates
$D(\bar H_1)$, and thus $C_0'=C_0$. For $x\in C$ with
$|xv|<\frac{\pi}{2}$, we have that $x\in H_1$. The result thus
follows.
\end{proof}

\begin{proof}[Proof of Theorem 4.6]
First by Lemma 4.7 and Remark 4.8, we get that $\dim(\Sigma_1)\leq
1$.
And by Sublemma 4.9, we have that $\Sigma_1$ is convex and isometric
to one of the following: $S^1(r)\ with\ r\leq 1$, $[ab]$,
$\{v_1,v_2\}\ with\ |v_1v_2|=\pi$, $\{v\}$.

We proceed the proof by induction. It is easy to see that for $n=2$
the theorem holds. Suppose $n-1$ the theorem holds.

In order to use induction, first we will prove the following:

\begin{sub lemma}
For every $ p\in \Sigma_0$ and every $w\in\Sigma_p\Sigma$, we have
that $|w\Sigma_p\Sigma_0|\leq\frac{\pi}{2}$ and
$|w(\Sigma_p\Sigma_0)^{\perp}|\leq\frac{\pi}{2}$.
\end{sub lemma}
\begin{proof}
Since $\Sigma_p\Sigma_0$ is convex without boundary in
$\Sigma_p\Sigma$, by Lemma 1.2 we get that
$|w\Sigma_p\Sigma_0|\leq\frac{\pi}{2}$. If $w\in \Sigma_p\Sigma$
such that $|w(\Sigma_p\Sigma_0)^{\perp}|>\frac{\pi}{2}$, then
$|w\Ua_p^{\Sigma_1}|>\frac{\pi}{2}$. Hence there exists $\uparrow
_p^q$ such that  $|\ua_p ^q\Ua_p^{\Sigma_1}|>\frac{\pi}{2}$, and
therefore $d_p\dist_{\Sigma_1}(\ua_p^q)>0$. Thus there exists $y\in
\Sigma$ such that $|y\Sigma_1|>\frac{\pi}{2}$, a contradiction to
the assumption of the theorem.
\end{proof}
We will prove the inductive steps according to the four situations
of $\Sigma_1$.\\

Case 1. Assume $\Sigma_1=S^1(r)$ with $r\leq 1$. By Lemma 4.15, we
can get that $\Sigma$ satisfies the conditions of Theorem 4.4. Then
the theorem
holds.\\

Case 2. Assume $\Sigma_1=\{v_1,v_2\}$ with $|v_1v_2|=\pi$. Then
$\Sigma=\Sigma_1*\Sigma_1^{\perp}$. Since $\Sigma$ is topologically
nice, $\Sigma_1^{\perp}$ is topologically nice and homeomorphic to a
sphere. Then by Lemma 4.17, $\Sigma_0$ is homeomorphic to $S^{n-2}$
and
topologically nice.\\

Case 3. Assume  $\Sigma_1=\{v\}$. Subcase 1. Assume  $\dim(\hat
\Sigma_1)=n$. For any $x\in\hat\Sigma_1$, on one hand, by the
assumptions of the theorem, $|xv|\leq \frac{\pi}{2}$, on the other
hand, by the definition of $\hat\Sigma_1$, $|xv|\geq\frac{\pi}{2}$.
Thus $|xv|=\frac{\pi}{2}$. By Lemma 4.7, we get a contradiction.

Subcase 2. Assume  $\dim(\hat \Sigma_1)=n-1$. Subsubcase 1. Assume
$\partial \hat \Sigma_1=\emptyset$. On one hand, by Lemma 4.17,
$\Sigma$ is homeomorphic to $S(\hat \Sigma_1)$, thus $\hat \Sigma_1$
separates $\Sigma$. On the other hand, for every $w\in \Sigma-(\hat
\Sigma_1\cup \Sigma_1)$, we have that $|w\hat
\Sigma_1|<\frac{\pi}{2}$ and $|w\Sigma_1|<\frac{\pi}{2}$, then by
triangle comparison, we get that $\dist_{\hat \Sigma_1}$ is
noncritical in $\Sigma-(\hat \Sigma_1\cup \Sigma_1)$, thus there is
a deformation retraction from $\Sigma-\hat \Sigma_1$ to $\{v\}$, a
contradiction.

Subsubcase 2. Assume $\partial \hat \Sigma_1\neq\emptyset$. We first
show that for any $x\in (\hat \Sigma_1)^{\circ}=(\hat
\Sigma_1-\partial \hat \Sigma_1)$, $\#\{\Ua_x^v\}=2$. By Lemma 4.17,
$\Sigma_x\Sigma\overset{homeo}{\simeq}S(\Sigma_x\hat\Sigma_1)$, and
by the proof of Lemma 4.19, we have that $\#\{\Ua_x^v\}\leq2$.
Suppose  $\#\{\Ua_x^v\}=1$, then by  Lemma 4.19, there is $w\in
\Sigma_x\Sigma$, such that $d_x\dist_v(w)>0$, a contradiction, as
can be seen in the proof of Sublemma 4.20.

By Lemma 4.16, we have that $\Sigma_0\cap\partial\hat
\Sigma_1\neq\emptyset$. If $\Sigma_0\cap\partial\hat \Sigma_1$ is
not equal to $\Sigma_0$, i.e., there is $z\in \Sigma_0\cap(\hat
\Sigma_1) ^{\circ}$, then by the above paragraph, we have that
$\#\{\Ua_z^v\}=2$. Thus for any $x\in\Sigma_0,\#\{\Ua_x^v\}=2$,
because by induction and the proof of Sublemma 4.11, we have that
for any $p,q\in \Sigma_0$, $\#\{\Ua_p^v\}=\#\{\Ua_q^v\}$. By the
proof of Sublemma 4.12, we can see that, $\psi:\Ua_v^{\hat
\Sigma_1^\circ}\to \hat \Sigma_1^\circ$, $\psi(\Ua_v^y)=y,y\in \hat
\Sigma_1^\circ$, is a locally isometric covering map. Since $\hat
\Sigma_1$ is contractible (the reason is that $\hat\Sigma_1\in
\Alex^{n-1}(1)$, with nonempty boundary, hence soul of  $\hat
\Sigma_1$ is a point (cf. \cite{[Per1]})),
$\Ua_v^{\hat\Sigma_1^\circ}\subset\Sigma_v\Sigma$ is two copies of
$\hat \Sigma_1^\circ$, by taking closure we can see that
$\Ua_v^{\Sigma_0}$ are two copies of $\Sigma_0$ with empty
intersection. This is impossible, since by applying the inductive
assumptions to $\Sigma_p\Sigma$ and $\Sigma_p\Sigma_0$, for $p\in
\Sigma_0$, we can see that $\Sigma_0$
is a topological manifold, and 
one component of $\Ua_v^{\Sigma_0}$ separates $\Sigma_v$, which is
homeomorphic to a sphere, into two components, with the closure of
each one convex. By Lemma 4.16, we can get a contradiction.

It follows that for any $p\in \Sigma_0$, $\#\{\Ua_p^v\}=1$. Hence
$\Sigma_0\subset
\partial\hat \Sigma_1$. We claim that $\partial\hat \Sigma_1 \overset
{homeo}{\simeq}S^{n-1}$. Therefore $\Sigma_0=\partial\hat \Sigma_1$.
It follows that $\Sigma$ contains two copies of $S^+(\hat
\Sigma_1)$, with each copy convex gluing along the boundary, which
is homeomorphic to $S^n$. Thus $\Sigma=[v\hat \Sigma_1]$, a double
of $v*\hat \Sigma_1$.

We now verify the claim: let $s\in \hat \Sigma_1$ be a point such
that $|s\partial \hat \Sigma_1|=\max\{\dist_{\partial\hat
\Sigma_1}\}$. By Lemma 4.17, $\Sigma_s\Sigma\overset
{homeo}{\simeq}S(\Sigma_s\hat \Sigma_1)$ and $\Sigma_s\hat
\Sigma_1\overset {homeo}{\simeq} S^{n-1}$. By 
Lemma 4.18, $(\hat\Sigma_1,\partial \hat \Sigma_1)\overset
{homeo}{\simeq}(\bar C(\Sigma_s\hat \Sigma_1),\Sigma_s\hat
\Sigma_1)$. Hence $\partial\hat \Sigma_1 \overset
{homeo}{\simeq}S^{n-1}$. 

Subcase 3. Assume  $\dim(\hat \Sigma_1)=n-2$. By Lemma 4.14, we get
that $\hat \Sigma_1= \Sigma_0$. Hence for every $w\in
\Sigma-(\Sigma_0\cup \Sigma_1)$,
$|w\Sigma_0|<\frac{\pi}{2},|w\Sigma_1|<\frac{\pi}{2}$. By triangle
comparison, we get that $\dist_{\Sigma_i}$ is noncritical in
$\Sigma-(\Sigma_0\cup \Sigma_1)$. Therefore there are deformation
retractions from $\Sigma-\Sigma_1$ to $\Sigma_0$, and
$\Sigma-\Sigma_0$ to $\Sigma_1$. Since for $p\in \Sigma_0$, by
induction, $\Sigma_p\Sigma_0$ is homeomorphic to a sphere, and thus
$\Sigma_0$ is a topological manifold. As in Theorem 4.4,
by using Alexander duality, we get a contradiction.\\

Case 4. Assume  $\Sigma_1=[ab]$. Subcase 1. Assume
$\dim(\hat\Sigma_1)=n$. Similarly as Subcase 1 of Case 3, we can get
a contradiction.

Subcase 2. Assume $\dim(\hat\Sigma_1)=n-1$. First by Lemma 4.7, we
have that $\partial \hat\Sigma_1\neq \emptyset$. By Remark 1.3, we
have that if there exist $x\in [ab]^{\circ}$ and $y\in
(\hat\Sigma_1)^{\circ}$, such that $|xy|=\frac{\pi}{2}$, then for
every $\bar x\in [ab]^{\circ}$ and every $\bar y\in \hat\Sigma_1,\
|\bar x\bar y|=\frac{\pi}{2}$. By Lemma 4.7, we get a contradiction.
Hence for every $x\in [ab]^{\circ}$ and every $y\in
(\hat\Sigma_1)^{\circ}$, $|xy|>\frac{\pi}{2}$. Since by Lemma 1.2,
we have that for every $x\in [ab]$ and for every $z\in \Sigma_0,\
|xz|=\frac{\pi}{2}$. It follows that $\Sigma_0\subset\partial
\hat\Sigma_1$.

Let $s\in \hat\Sigma_1$ be the point such that $|s\partial
\hat\Sigma_1|=\max\{\dist_{\partial \hat\Sigma_1}\}$. By Lemma 4.17,
$\Sigma_s\Sigma\overset {homeo}{\simeq}S(\Sigma_s\hat\Sigma_1)$, and
$\Sigma_s\hat\Sigma_1 \overset {homeo}{\simeq} S^{n-2}$. And by
Lemma 4.18, $(\hat\Sigma_1,\partial \hat\Sigma_1)\overset
{homeo}{\simeq}(\bar C(\Sigma_s\hat\Sigma_1),\Sigma_s\hat\Sigma_1)
$. Hence $\partial \hat\Sigma_1 \overset {homeo}{\simeq}S^{n-2}$. On
the other hand, by applying the inductive assumptions to
$\Sigma_p\Sigma$ and $\Sigma_p\Sigma_0$, for $p\in \Sigma_0$, we can
see that $\Sigma_0$ is a topological manifold. Hence
$\Sigma_0=\partial \hat\Sigma_1$.

Observe that for $z\in \Sigma_ 0$, by Lemma 1.2 and 2.6, we have
that $\Ua_z^{{\Sigma_1}^\circ}\overset{homeo}{\cong} (a,b)\times
\Ua_z^x$, where $x\in [ab]^{\circ}$. Since $\Ua_z^{\Sigma_1} \subset
(\Sigma_z\Sigma_0)^{\perp}$, whose dimension (by Lemma 4.7) $\leq
1$, we have that $\#\{[xz]\}$ is finite. By the proof of Sublemma
4.12, we can see that, $\psi_0:\Ua_x^{\Sigma_0}\to \Sigma_0$,
$\psi_0(\Ua_x^y)=y,y\in\Sigma_0$, is a locally isometric covering
map. Since $\Sigma_0\overset {homeo}{\simeq}S^{n-2}$,
$\Ua_x^{\Sigma_0}$ are several copies of $\Sigma_0$. On the other
hand,  suppose that $\Sigma_x\Sigma=S(\Sigma')$, then $\Sigma'$ is
homeomorphic to $S^{n-2}$ (by the topologically nice property of
$\Sigma$). Hence
$\Sigma'=\Ua_x^{\Sigma_0}\overset{isom}{\cong}\Sigma_0$, and
$\#\{[xz]\}=1$.
By now we get that $[\Sigma_0\Sigma_1]=\Sigma_0*\Sigma_1$, which can
be seen in the final part of the  proof of Theorem 4.4.

Finally, we need to show that for every $x\in
(\hat\Sigma_1)^{\circ}$, $|ax|=\frac{\pi}{2}$ and
$|bx|=\frac{\pi}{2}$. Since by the construction of $\hat \Sigma_1$
and by the condition of the theorem, $|x,[ab]|=\frac\pi2$, and we
have showed that $|x,[ab]^\circ|>\frac\pi2$, thus
$|ax|=\frac{\pi}{2}$ or $|bx|=\frac{\pi}{2}$, without loss of
generality, suppose that $|ax|=\frac\pi2$. We have that
$\#\{[ax]\}=1$. If not, as Subsubcase 2 in Case 3, we get that
$\Sigma=D(S^+(\hat\Sigma_1))$, a contradiction. Then by Lemma 4.17
and Lemma 4.19, $|\!\ua_x^a\ua_x^b\!|>\frac{\pi}{2}$. It follows
that $d_x\dist_{\Sigma_1}(\ua_x^b)>0$. This contradicts to the
condition of the theorem, as can be seen in the proof of Sublemma
4.20.

Thus $\Sigma$ contains $[\Sigma_0\Sigma_1]$ and two copies of
$S^+(\hat\Sigma_1)$, gluing along two copies of $S^+(\Sigma_0)$,
which is homeomorphic to
$S^n$. 
Hence $\Sigma=(a*\hat \Sigma_1\cup_{\hat\Sigma_1}
b*\hat\Sigma_1)\cup_\partial(\Sigma_1*\Sigma_0)$.

Subcase 3. Assume $\dim(\hat\Sigma_1)=n-2$. Similarly as the proof
of Subcase 3 of Case 3, we get a contradiction.
\end{proof}
Now we are ready to verify Proposition 2.1.
\begin{proof}[Proof of Proposition 2.1]
Observe that for any $p\in S$, we have that $\Ua_p^{\po}\subset
(\Sigma_pS)^{\perp}$, and by applying first variation formula to the
Busemann function, we get that for all $v\in\Sigma_pA,\
|v\Ua_p^{\po}|\leq\frac{\pi}{2}$. Hence
$|v(\Sigma_pS)^{\perp}|\leq\frac{\pi}{2}$. Thus for $p\in A$,
$\Sigma_pA$ satisfies the conditions of Theorem 4.6. Therefore
(2.1.1), (2.1.2) and (2.1.3) can all be derived from Theorem 4.6.

(2.1.4): As can be seen from the proof of Theorem 4.6 that for any
$x\in [ab]^\circ$, there is $y\in\hat\Sigma_1^p$ such that
$|xy|>\frac\pi2$. Then either $a$ or $b$ must be in $E$. If
$b\not\in E$, then also from the proof, there exists a point with
distance bigger than $\frac\pi2$ to $a$. The first statement thus
follows. The second statement is easy to be seen.
\end{proof}

\section*{Acknowledgement} First my sincere thanks will go to Prof.
Xiaochun Rong for many useful discussions during the preparation of
this work, especially for his vast time spent on the checking,
organization and presentation of this work. Then I will owe my
thanks to Prof. Jianguo Cao for his instructive work. I'm also in
great debt to Prof. Xuezhi Zhao for his unselfishness on topology
and many useful conversations. Furthermore, I would like to thank
Postdoc. Shicheng Xu for many constructive comments and
useful discussions. 
Finally I would like to thank Yusheng Wang, Xiaole
Su, Hongwei Sun and Chao Qian for their useful comments.\\

{\footnotesize Address: Department of Mathematics, Capital Normal
University, Beijing, China \newline\indent Email:
xuepingli525@gmail.com}


\end{document}